\documentclass[11pt]{article}
\usepackage{epsfig}
\usepackage{amssymb,amsmath,amsthm,amscd}
\usepackage{latexsym}
 \usepackage{bbm,dsfont}

\newtheorem{thm}{Theorem}[section]
\newtheorem{prop}[thm]{Proposition}

\newtheorem{lem}[thm]{Lemma}

\theoremstyle{definition}
\newtheorem{defn}[thm]{Definition}

\newcommand{\be}{\begin{equation}}
\newcommand{\ee}{\end{equation}}

\newcommand{\R}{\mathbb{R}}
\newcommand{\N}{\mathbb{N}}
\newcommand{\E}{\mathbb{E}}

\def \eps {{ \varepsilon }}

\def \o {{  {\mathbb{R}^n} }}

\def \calf {{  {\mathcal{F}} }}

\def \cala {{  {\mathcal{A}}  }}

\def \cale {{  {\mathcal{E}}  }}
\def \calg {{  {\mathcal{G}}  }}

\def \calz {{  {\mathcal{Z}}  }}
 \def \call {{  {\mathcal{L}}  }}

\begin{document}

\baselineskip=1.2 \baselineskip

%
%\begin{titlepage}
%\title{ \large \bf 
% Dynamics of 
%Fractional Stochastic Reaction-Diffusion Equations
%on Unbounded Domains Driven by Nonlinear Noise  }
%
%
% 
%\author{
% Bixiang Wang  
%\vspace{1mm}\\
%Department of Mathematics, New Mexico Institute of Mining and
%Technology \vspace{1mm}\\ Socorro,  NM~87801, USA \vspace{3mm}\\
%Email: bwang@nmt.edu\vspace{6mm}\\
%  }
% 
%
%
%
%\date{}
%\end{titlepage}

\begin{center}
 {  \bf 
 	Uniform
 Large Deviation Principles  of Fractional 
 Reaction-Diffusion  Equations
 Driven by  Superlinear  Multiplicative  Noise
 on $\R^n$
   }
\end{center}

\medskip

\medskip

\begin{center}
 Bixiang Wang  
\vspace{1mm}\\
Department of Mathematics, New Mexico Institute of Mining and
Technology \vspace{1mm}\\ Socorro,  NM~87801, USA \vspace{3mm}\\
Email: bwang@nmt.edu\vspace{6mm}\\
\end{center}

%\maketitle

%\medskip

\begin{abstract}  
In this paper, we investigate the uniform
large deviation principle of  
 the  fractional stochastic reaction-diffusion equation
on the entire space $\R^n$ as the noise intensity
approaches zero.
 The nonlinear drift term is dissipative
and has a polynomial growth of any order. The
  nonlinear diffusion term
  is locally Lipschitz continuous and has a 
  superlinear growth rate.
  By the weak convergence method, we establish
  the Freidlin-Wentzell uniform large deviations
  over bounded initial data as well as the 
  Dembo-Zeitouni uniform large deviations
  over compact initial data.
  The main difficulties are  caused by the
  superlinear growth of  noise coefficients
  and the non-compactness of   Sobolev 
  embeddings on unbounded domains.
  The dissipativeness of the drift term and the idea
  of uniform tail-ends estimates of solutions
  are employed to circumvent  these  difficulties.
   \end{abstract}



\section{Introduction}
\setcounter{equation}{0}

    This  paper is concerned with
    the  uniform  large deviation principle (LDP) of 
 the fractional stochastic reaction-diffusion
equation    with  a nonlinear drift term
of polynomial order driven by a superlinear
multiplicative noise on   $\R^n$, which is given by:
\be
  \label{intr1}
  du^{\eps} (t)
  + (-\Delta)^ \alpha  u^{\eps} (t)  dt
  + F(t,x, u^{\eps}(t)) dt
  =  
    g(t,x)   dt
  +\sqrt{\eps} \sigma (t, x,  u^{\eps}(t))    {dW} ,
  \ t>0,
  \ee 
  with   initial condition 
 \be\label{intr3}
 u^{\eps}( 0, x ) = u_0 (x), 
 \ee 
where   $x\in \R^n$,
$\eps \in (0,1)$ and   $g\in L^2_{loc}(\R; L^2(\o))$.
The operator 
 $(-\Delta)^\alpha$ with $\alpha
 \in (0,1)$ is  the fractional Laplacian.
 The nonlinear drift  $F$ has a 
 growth order $r$ for any $r>1$ in its last argument,
  the nonlinear diffusion 
   $\sigma$  has a superlinear growth 
   up to order 
   $ {\frac 34} + {\frac r4}$,
      and $W$
 is a
 cylindrical Wiener process  in $l^2$
 on
 a  complete filtered probability space 
$(\Omega, \mathcal{F}, 
   \{ { \mathcal{F}} _t\} _{t\ge 0},  P )$.

 Fractional  PDEs
 arise in many applications
    \cite{abe1,  garr1,    guan2,  
    jara1, kos1}. The well-posedness
    and the long term dynamics
    of such equations have been explored
    in   
   \cite{abe1,  caff1,  dine1,  gal1, garr1,  guan2,
    jara1, kos1,  luh1,   ros1, ser1, ser2}  
    and    
     \cite{chen1, gu1, luh2,  wan10, wangJDE2019},
     respectively.
     In the present paper, we will investigate
     the uniform LDP of  the fractional equation
      \eqref{intr1}-\eqref{intr3}
      with superlinear noise.
    
     The  LDP   of
  stochastic   equations  has been extensively
  studied 
  in the literature, see, e.g.,    
    \cite{bra1, card1, cerr1, che1,chow1,  dem1, 
  far1, fre1, fre2, gau1, kall1, mart1, pes1, sow1,
  str1,
  var1, var2, ver1}
  for the classical theory,  and
   \cite{bess1, brz1, bud1, bud2, cerr2, cerr3,
   cerr4, chu1, 
   duan1, dup1, liu1, ort1, ren1, roc1,sal1, sal2, sal3}
   for the weak convergence theory.
   In particular, 
   the reader is  further referred to 
   \cite{fre1, fre2} for the development 
  of  the classical  theory and to
   \cite{bud1,  dup1}   
  for  the weak convergence  theory.

    If $\alpha=1$, then the fractional equation
    \eqref{intr1} reduces to the standard
    reaction-diffusion equation.
    When the nonlinear drift  $F$  
     has  at most  a linear growth rate,
      the   
     LDP  of the standard  reaction-diffusion 
     equation has  been established  in
    \cite{cerr5, che1, chow1, fre1, kall1,
    pes1, sow1}.
    When  $F$  has a polynomial growth rate,
    the LDP of the   equation has been
    proved in 
     \cite{cerr1, cerr3, hai1, ren1}.

    On the other hand,   the 
      uniform  LDP    of stochastic equations
      has been  discussed  in   
    \cite{bis1, bud1, bud2, 
   cerr1, che1, dem1, gau1, fre2,
   pes1, sal1, sal2, sal3, ver1}, which is 
   crucial for  determining     the 
    exit time
   and exit place of    solutions.
   We remark that in all these publications,
   the nonlinear diffusion term
   for the reaction-diffusion equation
   has at most a {\it linear} growth, and
   the underlying domain of the equation  is {\it bounded},
   where the solution operators
  and   Sobolev embeddings are all compact.
   The compactness of   Sobolev embeddings 
   is critical for proving the tightness and
   convergence of a sequence of solutions as
   well as the compactness of level sets of  rate functions,
   which are all 
   needed  for establishing the
   LDP of solutions in bounded domains 
    when the nonlinear drift $F$ 
   has a polynomial growth.

   In the case of unbounded domains,
   Sobolev embeddings are no longer compact,
   which introduces an essential difficulty
   to prove the tightness of a set of solutions
   as well as the convergence of 
   a sequence of solutions with respect to
   some control parameters.
   Consequently, the LDP of 
   stochastic equations defined
   on unbounded domains
   cannot be obtained by the
   arguments developed  for bounded domains
   based on compactness of embeddings.

   Recently, 
   the LDP of the reaction-diffusion
   equation on $\R^n$ 
   with polynomial drift and
   globally Lipschitz diffusion  was
   proved  in \cite{wangJDE2023}, where the idea of
   uniform tail-ends estimates of solutions
   was used to circumvent the difficulty
   caused by the non-compactness of
   Sobolev embeddings on unbounded domains.
   However, the argument of 
   \cite{wangJDE2023}
   does not apply to the case
   of superlinear noise.
   It seems that  so far   there is no  
   result available
    in the literature 
    regarding  the LDP of   the reaction-diffusion equation 
    driven by a  superlinear multiplicative noise even if
    the domain is bounded.
    
     The goal of the present paper is to deal with
     this  issue  and prove the LDP
     (actually the uniform LDP)
     of \eqref{intr1}-\eqref{intr3}
     with polynomial  drift   and
     superlinear noise     
     on the entire space $\R^n$.
     
     The difficulty introduced by the superlinear noise
     lies in the fact that the derivation of uniform estimates
     of solutions is more involved in this case. 
     We will  take  advantage of the dissipativeness
     of the drift term $F$ to control the superlinear
     growth of the diffusion term $\sigma$ in \eqref{intr1}.
    Based on this idea, we 
    first  establish the
    existence and uniqueness of solutions of
    \eqref{intr1}-\eqref{intr3}
    by taking the limit of a sequence of
    approximate solutions defined in
    expanding  bounded  balls in $\R^n$
    (see Lemma \ref{le32}).
    We then derive the uniform estimates of solutions
    of the  control equation
    with respect to bounded initial data and controls,
    including the uniform tail-ends estimates of solutions
    (see Lemmas \ref{cosol} and \ref{tail}).
     The  uniform tail-ends estimates of solutions
      indicate that the solutions are uniformly small
      on the complement of a sufficiently large
      ball in $\R^n$, By this fact and   
      the compactness of   embeddings
      in bounded domains, we are able to
      prove the compactness of solutions
      of the control equation with respect
      to bounded controls in the entire space $\R^n$
      as in \cite{wangJDE2023}, which further
      implies the strong convergence of solutions
      with respect to the weak topology of controls
      in $L^2(0,T; l^2)$ (see Lemma \ref{wc}).
      Such strong convergence  plays a vital role
     for proving the uniform LDP of \eqref{intr1},
     including the compactness of level sets
     of rate functions.
     
     The first result of the paper is the
     Freidlin-Wentzell uniform  LDP
       of \eqref{intr1}-\eqref{intr3}
       in  the space
     $C([0,T],L^2(\R^n))
     \bigcap L^2(0,T; H^\alpha
     (\R^n))
     \bigcap L^p(0,T; L^p(\R^n))$ 
     with respect to  bounded initial data
     in $L^2(\R^n)$
       (see Theorem \ref{main}),
      which   is obtained by
     the uniform convergence of solutions
     of   the stochastic
     equation on bounded initial data
     as $\eps \to 0$
     (Lemma \ref{cvs}) along with the equivalence
     of the Freidlin-Wentzell uniform  LDP
     and the equicontinuous uniform Laplace principle.
     The concept of  equicontinuous uniform Laplace principle
     was introduced by Salins  in \cite{sal2} and the equivalence
     of  this principle and the 
      Freidlin-Wentzell uniform  LDP
      was proved in \cite[Theorem 2.10]{sal2}
      by the weak convergence method.

    The second result of the paper is the
     Dembo-Zeitouni uniform LDP
     of \eqref{intr1}-\eqref{intr3}
   with respect to compact initial data 
   in $L^2(\R^n)$.
   As shown in \cite{sal2},
   the     
     Freidlin-Wentzell uniform  LDP
is not equivalent to the 
      Dembo-Zeitouni uniform LDP
      over bounded initial data,
      Instead, they are equivalent
      over compact initial data under
      certain conditions \cite[Theorem 2.7]{sal2}.  
      We will  first prove  the 
      continuity  of 
       level sets of rate functions
      of \eqref{intr1} 
      in terms of the Hausdorff metric
      with respect to initial data, and
      then the
       Dembo-Zeitouni uniform LDP
       in  
     $C([0,T],L^2(\R^n))
     \bigcap L^2(0,T; H^\alpha
     (\R^n))
     \bigcap L^p(0,T; L^p(\R^n))$ 
     with respect to compact initial data
     in $L^2(\R^n)$
       (see Theorem \ref{main1}),
      Note that 
     if $\alpha =1$,    the  results of the paper
      are also valid  and the proof is simpler in this 
      case.

      We mention that 
   the  Freidlin-Wentzell 
     uniform LDP of the 
     reaction-diffusion equation 
       on $\R^n$ 
       with  
   additive noise was   established  in
   \cite{wangSD2023} by the
   uniform contraction  principle,
   which does not apply to the case of
   multiplicative noise.

     In the next section, we review 
     the uniform LDP of random systems.
     We then prove the well-posedness and the
     uniform LDP of 
       \eqref{intr1}-\eqref{intr3}
       in the last two sections.

 \section{Weak convergence theory 
 for uniform  large deviations}
 
 In this section, we 
 review basic  results of uniform LDP
 from the 
 weak convergence  theory   
  \cite{bud1, dup1, sal2}.

  Let  $l^2$ be the   space
of square summable   sequences of real numbers, 
and  $\{W(t)\}_{t\in [0,T]} $
   is 
a cylindrical Wiener process 
in   $l^2$
with respect to 
a  complete filtered probability space 
$(\Omega, \mathcal{F}, 
   \{ { \mathcal{F}} _t\} _{t\ge 0},  P )$.
  Then 
   there exists a   separable 
   Hilbert space $U$
   such that
    the embedding
  $l^2 \hookrightarrow U$ is   Hilbert-Schmidt
  and $W(t)$ takes values in $U$.
  Given a Banach space $Z$, the closed ball
  in $Z$ centered at the origin and of radius
  $r>0$ is denoted by
  $\overline{B}_r(Z)$.

Let  $(\cale, \rho)$  
and
$(\cale_0, \rho_0)$
be    Polish  spaces,  and
   $
 \{X^\eps_x: \eps>0, x\in \cale_0
 \}$ be a family of  $\cale$-valued random 
 variables.
  Suppose that for every
  $\eps>0$ and $x\in \cale_0$,
  there exists 
 a
 measurable map
 $\calg^\eps_x: C([0,T], U) \to \cale$ 
 such that
 $
 X^\eps_x = \calg ^\eps_x (W)$.
  Let $\cala$
 be the space of  
 $ l^2$-valued   processes
 $v$
 which are progressively measurable 
 and
 $\int_0^T \| v (t)\|^2 dt<\infty$
 $P$-almost surely.
 Given $N>0$, denote by
   \be\label{calan}
 \cala_N
 =\{v\in \cala: v (\omega) \in 
 \overline{B}_N(L^2(0,T; l^2))
 \ \text{for almost  all } \omega \in \Omega
 \},
\ee
   where 
   $\overline{B}_N(L^2(0,T; l^2))$
   is 
   the closed ball in 
   $L^2(0,T; l^2)$ with center zero and
   radius $N$, which is endowed with
     the weak topology.
     
     Given $x\in \cale_0$,
     let 
       $I_x : \cale \to [0, \infty]$
 be a  function such that
     for every
 $0\le s <\infty $, the level set
 $I^s_x =\{\phi \in \cale: I_x (\phi) \le s\}$ is a compact
 subset of $\cale$.
 Such a function $I_x$ is called a good rate function
 on $\cale$.

 The Freidlin-Wentzell uniform LDP
 is stated  below \cite[Section 3.3]{fre2}.
 
 \begin{defn}
 The family $\{X^\eps_x, \eps>0,
 x\in \cale_0 \}$ is said to satisfy
  the
 Freidlin-Wentzell uniform LDP
 in $\cale$ with  
 rate functions $I_x$ 
 uniformly on $\cale_0$ if  
 
 (i)  For every $s>0$  and $\delta>0$,
 $$
 \liminf_{\eps \to 0}
 \inf_{x\in \cale_0}
 \inf_{\phi
 \in I_x^{s}}
 \left (
 \eps \ln P
 \left (
 \rho (X_x^\eps, \
 \phi  )
 < \delta
 \right )
 +I_x (\phi)
 \right )
 \ge 0.
 $$
 
 (ii)  For every $s_0>0$  and
 $\delta>0$,
 $$
 \limsup_{\eps \to 0}
 \sup_{x\in \cale_0}
 \sup_{0\le s\le s_0} 
 \left (
 \eps \ln P
 \left (
 {\rm dist} (X_x^\eps, \
 I_x^s   )
\ge \delta
 \right ) +s  
 \right )
\le  0.
 $$
 \end{defn}

 The next concept is
 the  Dembo-Zeitouni
   uniform LDP  \cite[Corollary 5.6.15]{dem1}.

    \begin{defn}
 The family $\{X^\eps_x, \eps>0,
 x\in \cale_0 \}$ is said to satisfy
  the
 Dembo-Zeitouni
  uniform LDP
 in $\cale$ with  
 rate functions $I_x$ 
 uniformly on $\cale_0$ if  
 
 (i)  For every open subset $G$ in $\cale$,
 $$
 \liminf_{\eps \to 0}
 \inf_{x\in \cale_0}
  \left (
 \eps \ln P
 \left ( 
  X_x^\eps \in G \right )
  \right )
  \ge -\sup_{x\in \cale_0}
 \inf_{\phi\in G}
  I_x(\phi) .
 $$
 
 (ii)  For every closed subset $F$ in $\cale$,
  $$
 \limsup_{\eps \to 0}
 \sup_{x\in \cale_0} 
 \left (
 \eps \ln P
 \left (  X_x^\eps \in F
 \right )\right )
 \le - \inf_{x\in
 \cale_0}
 \inf_{\phi\in F}
 I_x (\phi).
  $$
 \end{defn}

    It follows from
    \cite{sal2} that
    the Freidlin-Wentzell
    uniform LDP is not
    equivalent to the
    Dembo-Zeitouni
    uniform LDP in general.
    But they are equivalent
    if $(\cale_0, \rho_0)$ is
    compact and the
    the level sets of $I_x$
    are continuous with respect to
    $x\in \cale_0$ in
    terms  of the
    Hausdorff metric.
    
     \begin{prop}\label{eqldp}
     \cite[Theorem 2.7]{sal2}
     Suppose 
     $(\cale_0, \rho_0)$ is
    compact and  for every $s\ge 0$,
    $$
    \lim_{n\to \infty}
    \max
    \left \{
    \sup_{\phi
    \in I^s_x}
    {\rm dist}
    (\phi,
    I_{x_n}^s),\
    \sup_{\phi\in I^s_{x_n}}
    {\rm dist}
    (\phi, I^s_x)
    \right \}
    =0,
    \ \ 
    \rm{ as \ }  
    \rho_0 (x_n ,x )  \to 0.
    $$
    Then the  
     family $\{X^\eps_x, \eps>0,
 x\in \cale_0 \}$   satisfies 
 the Freidlin-Wentzell
 uniform LDP in  $\cale$
 with  
 rate functions $I_x$ 
 uniformly on $\cale_0$ 
 if and only if it satisfies 
  the
 Dembo-Zeitouni
  uniform LDP
 in $\cale$ with  
 rate functions $I_x$ 
 uniformly on $\cale_0$. 
  \end{prop}

    In order to obtain
      the Freidlin-Wentzell
      uniform  LDP 
   of $\{X^\eps_x,
   \eps>0, x\in \cale_0\}$ by the weak convergence method, 
   we  assume  that  
  there exists a measurable map
  $\calg^0_x:    C([0,T], U) \to \cale$
  such that for every $N>0$ 
  and $\delta>0$,
 \be\label{convp}
 \lim_{\eps\to 0}
 \sup_{x\in \cale_0}
 \sup_{v\in 
 \cala_N}
 P
 \left (
 \rho
 \left (
 \calg_x^\eps \left (
 W+
 \eps^{-\frac 12}\int_0^\cdot
 v(s) ds
 \right ),
 \
 \calg_x^0
 \left (
 \int_0^\cdot v(s) ds
 \right )
 \right ) >\delta
  \right )
  =0.
 \ee
 In addition, we assume that
   for every $N<\infty$
   and $x\in \cale_0$, the
  set  
  \be\label{convp1}
  \left \{
  \calg^0_x  
  \left (\int_0^{\cdot} v(t) dt
  \right ):\
  v \in \overline{B}_N(L^2(0,T; l^2))
  \right \}
  \ \text{ 
  is a compact subset of }
  \  \cale.
  \ee
  For every $x\in \cale_0$,
  define $I_x: \cale \to [0, \infty]$ by,
  for any  $\phi \in \cale$,
 \be\label{convp2}
  I_x(\phi)= \inf 
  \left \{
  {\frac 12} \int_0^T \| v(t)\|^2_{l^2} dt:\
   v\in L^2(0,T;  {l^2})
   \ \text{such that}\  
  \calg^0_x \left (\int_0^\cdot v(t) dt\right ) =
  \phi
  \right \},
 \ee
  with the convention that
  the infimum of the  empty set
   is   $+\infty$.
   
   It follows from \eqref{convp1}
   that $I_x$  given by \eqref{convp2}
   is a good rate function on $\cale$.

   \begin{prop}\label{fwuldp}
   \cite[Theorem 2.13]{sal2}
   If \eqref{convp} and \eqref{convp1}
   hold, then  the  
    family $\{X^\eps_x, \eps>0,
 x\in \cale_0 \}$   satisfies 
 the Freidlin-Wentzell
 uniform LDP in  $\cale$
 with  
 rate functions $I_x$  given
 by \eqref{convp2} 
 uniformly on $\cale_0$. 
   \end{prop}

   We will utilize Proposition
   \ref{fwuldp} to establish
   the uniform LDP
   of the stochastic equation
   \eqref{intr1}-\eqref{intr3}
   in the space
   $C([0,T], H)
   \bigcap L^2(0,T; V)
   \bigcap L^p
   (0,T; L^p(\R^n))$
   in the last section.

\section{Existence and uniqueness of solutions} 
\setcounter{equation}{0}

This section is devoted
to the existence and uniqueness
of solutions of \eqref{intr1}-\eqref{intr3}
under certain assumptions on the
nonlinear drift $F$  and the nonlinear
diffusion $\sigma$.

 Given $\alpha\in (0,1)$
 and $u\in L^2(\R^n)$, denote by  
$$
(-\Delta )^\alpha u
= 
  {\mathcal{F}^{-1}}
(|\xi|^{2 \alpha} (\mathcal{F} u)), \quad \xi \in \R^n,
$$
where
${\mathcal{F}}$    and 
${\mathcal{F}}^{-1}$ are 
the Fourier transform   and the 
inverse 
 Fourier transform, respectively.
  Let
  $H^ \alpha (\R^n)$ 
 be the 
  the  fractional Sobolev space
   as given by
  $$
 H^ \alpha (\R^n)
 =\left \{
 u\in L^2(\R^n): (-\Delta)^{\frac {\alpha}2}
 u \in L^2(\R^n) \right \},
  $$
  with norm
  $$
 \| u\|_{H^ \alpha (\R^n)}^2
 = \| u\|^2_{L^2(\R^n)}
 +   \| (-\Delta)^{{\frac  \alpha2}} u \|^2_{L^2(\R^n)}, 
$$
and  inner product
 $$
 (u, v)_{H^ \alpha (\R^n)}
 =
  \int_{\R^n} u(x)v(x)  dx 
  +
  \int_{\R^n}
  \left ( (-\Delta)^{{\frac  \alpha2}} u (x)
  \right )
  \left (
    (-\Delta)^{{\frac  \alpha2}} v (x)
    \right ) dx,
    $$ 
  for all
  $u, v\in H^\alpha
  (\R^n)$.
  By 
    \cite{dine1}  we find that
 the norm and the  inner product of
$H^ \alpha  ({\R^n})$ can be  rewritten as: 
   $$
 \| u\|^2_{H^ \alpha ({\R^n}) }
 = 
 \int_{\R^n} |u(x)|^2 dx 
  +
  {\frac {1}2}C(n,\alpha)
  \int_{\R^n} \int_{\R^n}
 {\frac {|u(x)- u(y)|^2} {|x-y|^{n+2 \alpha} }} dxdy , 
$$
and   
 $$
 (u, v)_{H^ \alpha (\R^n)}
 =
  \int_{\R^n} u(x)v(x)  dx 
  +  {\frac {1}2}C(n,\alpha)
  \int_{\R^n} \int_{\R^n}
 {\frac {(u(x)- u(y)) (v(x) -v(y))} {|x-y|^{n+2 \alpha} }} dxdy, 
 $$
for
  all $u, v \in H^ \alpha  ({\R^n})$, 
  where  
    $C(n,\alpha) 
 =\frac{\alpha 4^\alpha\Gamma(\frac{n+2\alpha}{2})}
{\pi^\frac{n}{2}\Gamma(1-\alpha)}
$ is a positive number.

 For convenience, we write  
$H= 
  L^2(\R^n) $ 
  with norm  
  $\| \cdot \|$  and 
  inner product $(\cdot, \cdot)$.
  Denote by 
$V= 
H^ \alpha (\R^n) $.
Then we have
$V \hookrightarrow H  = H^* \hookrightarrow V^*$
where $H^*$   and $V^*$ are the dual spaces
of $H$  and $V$, respectively.
 The space of Hilbert-Schmidt
   operators  between   
   separable 
   Hilbert spaces  $H_1$ and $H_2$
   is written
   as   $\call_2(H_1,H_2)$,
   and the corresponding norm
   is written as 
   $\| \cdot \|_{\call_2(H_1,H_2)}$.

      We now discuss 
      assumptions on the nonlinear drift
      $F$ in \eqref{intr1}.
We  assume  
$F: \R \times \o \times \R$
$\to \R$ is      continuous     
 such that  for all
$t, u, u_1, u_2  \in \R$   and  $x \in \o$, 
 \be 
\label{f1}
F (t, x, u) u
\ge \lambda_1 |u|^p -\psi_1(t,x),
\ee
\be 
\label{f2}
|F(t, x, u) |   \le
 \ \psi_2 (t,x)
  |u|^{p-1}  + \psi_3 (t,x),
\ee
$$
(F(t,x, u_1)
-F(t,x, u_2))(u_1-u_2) 
$$
\be 
\label{f3}
\ge
\lambda_2
(|u_1|^{p-2}u_1-|u_2|^{p-2}u_2)(u_1-u_2)
- \psi_4(t,x )|u_1-u_2|^2,  
\ee
where  $\lambda_1>0$,
  $\lambda_2>0$ and
   $ p>  2$ are constants, 
$\psi_1 \in L^1_{loc} (\R, L^1(\R^n))$,
$\psi_2, \psi_4  \in L_{loc}^\infty 
(\R, L^\infty(\R^n))$, and
$\psi_3\in 
  L_{loc}^{\frac p{p-1}}(\R, L^
  {\frac p{p-1}}
  (\R^n))$.
  
       To deal with the nonlinearity $F$, we
       will  frequently
       use the following  elementary inequalities:
       for any $u_1, u_2\in \R$ and $p>2$,
       \be
       \label{elity1}
        \left (
       |u_1|^{p-2} u_1
       - |u_2|^{p-2} u_2
       \right )
       \left (u_1
       -u_2 \right )
       \ge
       2^{1-p}  | u_1-u_2 |^p,
       \ee
       and
        \be
       \label{elity2}
        \left (
       |u_1| ^{p-2} u_1
       - |u_2|^{p-2} u_2
       \right )
       \left (u_1
       -u_2 \right )
       \ge
       {\frac 12}
       \left (
       |u_1|^{p-2}
       +|u_2|^{p-2}
       \right )
        \left (u_1-u_2  \right )^2 .
       \ee

Let
        $\sigma:
      \R \times \R^n \times \R
      \to l^2$  be a mapping as defined by   
  \be\label{sig1}
     \sigma (t,x, u)
     = \sigma_1  (t,x)
     +   
       \kappa (x)   \sigma_{2} ( t,x, u)  
        , 
       \quad \forall \ t\in \R, \ x\in \R^n, \ u\in \R,
\ee
     where     
   $\kappa \in    L^2(\R^n) \bigcap L^\infty(\R^n)$,
   $\sigma_1\in L^\infty_{loc}
   (\R,  L^2(\R^n, l^2 ))$,
   and 
    $ 
     \sigma_2 (t,x, u)
     =\left \{ 
        \sigma_{2,k} ( t,x, u)  \right \}_{k=1}^\infty
      $ for $t\in \R$,
      $x\in \R^n$ and $u\in \R$.
      We    further assume that  
    for each $k\in \N$,
    $\sigma_{2,k}: \R \times \R^n \times \R
    \to \R$ is measurable
    and 
   there exist  positive numbers $\alpha_k$, $\beta_k$ and $\gamma_k$
   such that   for all 
   $t\in \R$,
   $x\in \R^n$  and $u, u_1, u_2 \in \R$,
    \be\label{sig2}
   |\sigma_{2,k} (t,x, u_1) -\sigma_{2,k} (t,x,u_2) |^2
   \le \alpha_k \left (1 + |u_1|^{q-2}
   + |u_2|^{q-2} \right ) |u_1  -u_2 |^2,
   \ee
   \be\label{sig3}
   |\sigma_{2,k}  (t,x,u)  |^2
   \le \beta_k    +\gamma _k  |u|^q,
   \ee
   where $q\in [2,  1+{\frac p2}]$ and
  \be\label{sig4}
  \sum_{k=1}^\infty  ( \alpha_k
   +\beta_k  + \gamma_k  ) <\infty .
     \ee
     
     Note that  $\sigma_{2,k}$ has a  superlinear growth up to order ${\frac 12} + {\frac p4}$
     in its last argument for every $k\in \N$.
       Let 
     $\sigma_1 (t,x)
     =\left \{ \sigma_{1,k} (t,x)\right \}_{k=1}^\infty
     \in l^2$
     for $t\in \R$ and  $x\in \R^n$.
     Then by \eqref{sig1}  we have
    \be\label{sig5}
     \sigma   (t,x,u) 
     = 
     \left \{ \sigma_{1,k}  (t,x)
     +   
       \kappa (x)   \sigma_{2,k} ( t,x, u)  
       \right \}_{k=1}^\infty
        , 
       \quad \forall \ t\in \R, \ x\in \R^n, \ u \in \R.
  \ee
  By \eqref{sig3} and \eqref{sig5} we infer that
 for every $\eps>0$,   $u\in L^p(\R^n)$ and $t\in \R$,  
  $$
  \| \sigma (t, \cdot, u)\|^2_{L^2(\R^n, l^2)}
  =\sum_{k=1}^\infty
  \int_{\R^n}
  \left |  \sigma_{1,k}  (t,x)
     +   
       \kappa (x)   \sigma_{2,k} ( t,x, u(x) ) 
       \right |^2 dx
  $$
 $$
  \le
  2 \sum_{k=1}^\infty 
  \|\sigma_{1, k} (t) \|^2
  + 2 
  \sum_{k=1}^\infty \beta_k \| \kappa \|^2
  +2\sum_{k=1}^\infty
  \gamma_k\int_{\R^n} \kappa^2(x) |u(x)|^q dx
 $$
 \be\label{sig5a}
  \le
  2 \sum_{k=1}^\infty 
  \|\sigma_{1, k} (t) \|^2
  + 2 
  \sum_{k=1}^\infty \beta_k \| \kappa \|^2
  +2\sum_{k=1}^\infty
  \gamma_k \|\kappa\|_{L^{\frac {2p}{p-q}} 
  (\R^n) }^2
  \|u \|^q_{L^p(\R^n)} 
\ee
 \be\label{sig6}
  \le
  \eps  \|u \|^p_{L^p(\R^n)} 
  +
  2 \|\sigma_1 (t)\|_{L^2(\R^n, l^2)}^2
   +C(\eps),
  \ee
  where
  $C(\eps)>0$ depends only
  on $\eps$, $p, q$, $\kappa$, $\sum_{k=1}^\infty
  {\beta_k}$
  and $\sum_{k=1}^\infty \gamma_k$.
  
  On the other hand, 
 by \eqref{sig2}
 and \eqref{sig5}, we find that for every 
   $\eps>0$, 
 $u_1,u_2\in
L^2(\mathbb{R}^n)\cap L^p(\mathbb{R}^n)$
and $t\in \R$,
$$
 \|\sigma(t, \cdot, u_1)-
\sigma(t, \cdot,u_2)\|^2_{L^2(\mathbb{R}^n,l^2)}
 = \sum_{k=1}^\infty 
 \int_{\mathbb{R}^n}
|\kappa (x)|^2|\sigma_{2,k} (t,
x,u_1(x))-\sigma_{2,k}(t, x,u_2(x))|^2dx
$$
$$
\leq   \sum_{k=1}^\infty \alpha_k
\|\kappa \|^2_{L^\infty (\R^n)}
 \int_{\mathbb{R}^n} 
   \big( 1+ |u_1(x)|^{q-2}  
+|u_2(x)|^{q-2}
\big)|u_1(x)-u_2(x)|^2dx
 $$
 \be\label{sig7}
 \le
 C(\eps)
 \|u_1-u_2\|^2
 +\eps
 \int_{\mathbb{R}^n} 
   \big(   |u_1(x)|^{p-2}  
+|u_2(x)|^{p-2}
\big)|u_1(x)-u_2(x)|^2dx,
\ee
 where
  $C(\eps)>0$ depends only
  on $\eps$, $p, q$, $\kappa$ and  $\sum_{k=1}^\infty
  {\alpha_k}$.
  
  Similarly, we see that 
 for every    $u\in L^2(\R^n)
 \bigcap  L^p(\R^n)$ and $t\in \R$,  
  $$
  \| \sigma (t, \cdot, u)\|^2_{L^2(\R^n, l^2)}
  \le
  2 \sum_{k=1}^\infty 
  \|\sigma_{1, k} (t) \|^2
  + 2 
  \sum_{k=1}^\infty \beta_k \| \kappa \|^2
  +2\sum_{k=1}^\infty
  \gamma_k\int_{\R^n} \kappa^2(x) |u(x)|^q dx
 $$
 $$
  \le
  2 \sum_{k=1}^\infty 
  \|\sigma_{1, k} (t) \|^2
  + 2 
  \sum_{k=1}^\infty \beta_k \| \kappa \|^2
  +2\sum_{k=1}^\infty
  \gamma_k \|\kappa\|_{\frac {4p}{p- 2(q-1) }}^2
  \|u \|^{q-1} _{L^p(\R^n)}  \| u \|
$$
$$
 \le
  2 \sum_{k=1}^\infty 
  \|\sigma_{1, k} (t) \|^2
  + 2 
  \sum_{k=1}^\infty \beta_k \| \kappa \|^2
  +2\sum_{k=1}^\infty
  \gamma_k \|\kappa\|_{\frac {4p}{p- 2(q-1) }}^2
  \left ( {\frac {p-2(q-1)}p}
  + {\frac {2(q-1)}p}
  \|u \|^{\frac p2} _{L^p(\R^n)}  
  \right )\| u \|
  $$
   \be\label{sig8}
  \le
   2 \|\sigma_1 (t)\|_{L^2(\R^n, l^2)}^2
   +  2 
  \sum_{k=1}^\infty \beta_k \| \kappa \|^2
   +C 
   \left (
   1+    \|u \|^{\frac p2} _{L^p(\R^n)} 
   \right ) \| u \|,
  \ee
  where
  $C >0$ depends only
  on   $p, q$, $\kappa$ 
  and $\sum_{k=1}^\infty \gamma_k$.
  
  Note that  for any
 $u_1,u_2\in
L^2(\mathbb{R}^n)\cap L^p(\mathbb{R}^n)$
and $t\in \R$,
$$
 \|\sigma(t, \cdot, u_1)-
\sigma(t, \cdot,u_2)\|^2_{L^2(\mathbb{R}^n,l^2)} 
\le   \sum_{k=1}^\infty \alpha_k
\int_{\R^n} \kappa^2 (x)
 |u_1(x)  -u_2(x)  |^2  dx
$$
$$
+  \sum_{k=1}^\infty \alpha_k 
 \int_{\mathbb{R}^n} \kappa^2 (x)
   \big(   |u_1(x)|^{q-1}  
+|u_2(x)|^{q-1}
+ |u_1 | \ | u_2|^{q-2}
 + |u_2 | \ | u_1|^{q-2}
\big)|u_1-u_2|dx
 $$
 $$
 \le   \sum_{k=1}^\infty \alpha_k
 \int_{\R^n} \kappa^2 (x)
 \left ( |u_1(x) |
 + |  u_2(x)  | \right ) |u_1(x)
 -u_2 (x)| dx
$$
$$
+   2\sum_{k=1}^\infty \alpha_k 
 \int_{\mathbb{R}^n}  \kappa^2 (x) 
   \big(   |u_1(x)|^{q-1}  
+|u_2(x)|^{q-1} 
\big)|u_1(x)-u_2(x) |dx
 $$
 $$
 \le   \sum_{k=1}^\infty \alpha_k 
 \left ( \| \kappa^2(\cdot) u_1 \|
 +
 \| \kappa^2(\cdot) u_2 \|
 \right )  \| u_1 
 -u_2 \|
$$
$$
+   2\sum_{k=1}^\infty \alpha_k 
 \left (
 \|  \kappa^2 (\cdot )  
    |u_1 |^{q-1}  \|
    + \|  \kappa^2  (\cdot) 
    |u_2 |^{q-1}  \|
    \right )  
 \|u_1-u_2  \|
 $$
  $$
 \le   \sum_{k=1}^\infty \alpha_k
 \|\kappa\|^2_{\frac {4p}{p-2}}
 \left (\|u_1\|_{L^p(\R^n)}
 +
 \|u_2\|_{L^p(\R^n)}
 \right )  \|u_1 
 -u_2  \|
$$
$$
+   2\sum_{k=1}^\infty \alpha_k 
\|\kappa\|^2 _{L^{\frac {4p}
{p-2(q-1)}} (\R^n)}
 \left ( 
    \| u_1 \|^{q-1} _{L^p(\R^n)}
    + 
    \| u_2 \|^{q-1} _{L^p(\R^n)}  
    \right )  
 \|u_1-u_2  \|
 $$
  $$
 \le   2p^{-1} \sum_{k=1}^\infty \alpha_k
  \|\kappa\|^2_{\frac {4p}{p-2}} 
 \left (  p-2
 + 
 \| u_1 \|_{L^p(\R^n)}^{\frac p2}
 + 
 \| u_2 \|_{L^p(\R^n)}^{\frac p2}
  \right ) \| u_1 
 -u_2 \|
$$
$$
+  {\frac  
{4(q-1)}p} 
\sum_{k=1}^\infty \alpha_k 
\|\kappa\|^2 _{L^{\frac {4p}
{p-2(q-1)}} (\R^n)}
 \left (  {\frac {p-2(q-1)}{q-1}}   
 +
     \| u_1 \|^{\frac p2} _{L^p(\R^n)}
    + 
    \| u_2 \|^{\frac p2} _{L^p(\R^n)}  
    \right )  
 \|u_1-u_2  \|
 $$
 \be\label{sig9}
  \le C 
   \left ( 1
 + 
 \| u_1 \|_{L^p(\R^n)}^{\frac p2}
 + 
 \| u_2 \|_{L^p(\R^n)}^{\frac p2}
  \right ) \| u_1 
 -u_2 \|,
\ee
 where
  $C >0$ depends only
  on   $p, q$, $\kappa$ and  $\sum_{k=1}^\infty
  {\alpha_k}$.

  Given $u\in L^p(\R^n)$  and $t\in \R$, define an operator
  $\sigma (t, u): l^2\to H$ by
\be\label{sig10}
  \sigma (t,u) (v) (x)
  =\sum_{k=1}^\infty
  \left (
   \sigma_{1,k}  (t,x)
     +   
       \kappa (x)   \sigma_{2,k} ( t,x, u(x) )
       \right )
         v_k,
  \quad \forall \ 
   v=\{v_k\}_{k=1}^\infty \in l^2,
   \ \ x\in \R^n.
\ee
By \eqref{sig6}  we find that
for every $u\in L^p(\R^n)$ and $t\in \R$,
  $\sigma (t, u): l^2\to H$ is a   Hilbert-Schmidt 
  operator 
   with norm
 $$
 \|  \sigma (t,u)   \|^2_{\call_2 (l^2, H)}
 =
   \| \sigma (t, \cdot, u)\|^2_{L^2(\R^n, l^2)} .
$$ 
In addition, we have
for any $u_1, u_2\in L^p(\R^n)$ and $t\in \R$, 
  $$
 \|  \sigma (t,u_1) -
 \sigma (t,u_2)  
   \|^2_{\call_2 (l^2, H)}
 =\|  \sigma (t, \cdot, u_1) -
 \sigma (t, \cdot, u_2)   \|^2_{L^2(\R^n, l^2)} .
$$

%  
%  Throughout this paper,  we   assume that 
%   $(\Omega, \mathcal{F},
%    \{ { \mathcal{F}} _t\} _{t\in [0,T]},  P  )$
%    with $T>0$ 
%is  a  complete filtered probability space
% satisfying the usual condition,
% and 
%  $\{W(t)\}_{t\in [0,T]}$  
%   is a cylindrical Wiener
%process in $l^2$
%with identity covariance 
%operator.
%Then 
%there exists a separable Hilbert space
%$U$ such that the embedding
%$l^2 \hookrightarrow U$ is Hilbert-Schmidit and
%$W$ is a $U$-valued Wiener process.
%  
%  

  By the definition of  the Hilbert-Schmidt
  operator $\sigma$, we can reformulate
   system \eqref{intr1}-\eqref{intr3}
   as follows:
  \be
  \label{sys1}
  du^{\eps} (t)
  + (-\Delta)^ \alpha  u^{\eps} (t)  dt
  + F(t,x, u^{\eps}(t)) dt
  $$
  $$
  =  
    g(t,x)   dt
  +\sqrt{\eps} 
  \sum_{k=1}^\infty
  \left (
  \sigma_{1,k}  (t,x)
     +   
       \kappa (x)   \sigma_{2,k} ( t,x, u^{\eps}(t) ) 
       \right ) dW_k
   ,
  \ee 
  with   initial condition 
 \be\label{sys2}
 u^{\eps}( 0, x ) = u_0 (x),   \quad x\in \R^n,
 \ee 
  where $\{W_k\}_{k=1}^\infty$
  is a sequence of independent  real-valued
  standard 
  Wiener processes.

 \begin{defn}
 \label{defnsol}
Suppose 
      $u_0 \in L^2(\Omega, \calf_0; H)$
      and $u^{\eps}$ is 
   a continuous $H$-valued $\calf_t$-adapted
   process such that 
\be\label{defnsol_1}
 u^{\eps}\in L^2(\Omega, C([0, T], H))
 \bigcap L^2(\Omega, L^2(0, T; V))
 \bigcap L^p(\Omega, L^p(0, T; L^p(\R^n))).
\ee
  If  for  all $t\in [0,T]$
 and $\xi \in V\bigcap L^p(\R^n)$,
 $P$-almost surely,
\be\label{defnsol_2}
 ( u^{\eps}(t), \xi )
 + \int_0^t
 ( (-\Delta )^{\frac \alpha 2} u^{\eps}(s) , (-\Delta )^{\frac \alpha 2} \xi) ds
 +
 \int_0^t\int_\o F(s, x, u^{\eps}(s)) \xi (x) dx  ds
 $$
 $$
 =(u_0, \xi)
    +\int_0^t (g(s),  \xi  ) ds
  +\sqrt{\eps} \int_0^t
  ( \xi , 
  \sigma (s, u^{\eps}(s)) dW ),
\ee
then $u$ is called a solution of \eqref{intr1}-\eqref{intr3}.
 \end{defn}
 
 Notice that 
   if $u^{\eps}$ is a solution of \eqref{intr1}-\eqref{intr3}
   in the sense of Definition \ref{defnsol},
  then  
  for all  $t\in [0,T]$,  
\be\label{defnsol_3}
   u^{\eps}(t) 
 + \int_0^t
   (-\Delta )^{  \alpha  } u^{\eps}(s)  ds
 +
 \int_0^t  F(s, \cdot , u^{\eps}(s))    ds
 =u_0
   +\int_0^t  g(s)  ds
  + \sqrt{\eps} \int_0^t
  \sigma (s, u^{\eps}(s)) dW
\ee
 in $ (V\bigcap L^p(\R^n) )^*$, $P$-almost surely.

 The existence and uniqueness of solutions of
 \eqref{intr1}-\eqref{intr3}
 is given below.

 \begin{lem} \label{le32}
 If
 \eqref{f1}-\eqref{f3}
  and \eqref{sig1}-\eqref{sig4}
  hold, then for every
    $T>0$,
  $\eps\in (0,1)$
  and  $u_0\in L^2(\Omega,\calf_0; H)$,
  there exists a unique solution
  $u^\eps$  to   \eqref{intr1}-\eqref{intr3}
 such that 
 $$
\| u^{\eps}  \|_{L^2(\Omega, C([0, T], H) ) }^2
+
\| u ^{\eps}  \|_{L^2 (\Omega, L^2(0, T;  V ) )}^2
+
\| u ^{\eps} \| _{L^p(\Omega, L^p(0, T;  L^p(\R^n) ) )}^p
\le L_1
\left (1+ 
\| u_0 \|^2_{L^2(\Omega, H)} 
\right ),
$$
where $L_1>0$ is a constant
depending only on $T$,  but not on
 $u_0$ or $\eps$.
 \end{lem}
  
  \begin{proof}
  Given
  $k\in \N$,
   let $\mathcal{O}_{k}=\{x\in\R^n: |x|< k\}$.
   We first consider the  solutions of the stochastic equation defined on
   $\mathcal{O}_k$ and
   then take the limit of the solutions
   as $k\to \infty$ to obtain a solution
   defined on the entire space $\R^n$.
   
   For every 
   $u_0\in L^2(\Omega,\calf_0; H)$ and 
    $k\in \N$, 
   let $u_{0,k} = u_0$ for
   $|x|<k$,  and $u_{0,k} =  0$ otherwise.
   Then consider the 
   following stochastic  problem
   on $\mathcal{O}_k$ 
   as an approximation of \eqref{sys1}-\eqref{sys2}:
   \be
  \label{sys1a}
  du^{\eps}_k (t)
  + (-\Delta)^ \alpha  u^{\eps}_k (t)  dt
  + F(t,x, u^{\eps}_k(t)) dt
  $$
  $$
  =  
    g(t,x)   dt
  +\sqrt{\eps} 
  \sum_{k=1}^\infty
  \left (
  \sigma_{1,k}  (t,x)
     +   
       \kappa (x)   \sigma_{2,k} ( t,x, u^{\eps}_k(t) ) 
       \right ) dW_k
   ,
  \ee 
  with boundary condition
  \be\label{sys2a}
 u^{\eps}_k( t, x ) = 0,   \quad x\in \R^n\setminus {\mathcal{O}}_k,
 \  t>0,
 \ee 
  and    initial condition 
 \be\label{sys3a}
 u^{\eps}_k( 0, x ) = u_{0,k} (x),   \quad x\in \R^n.
 \ee 
   
   Let $H_k$, $V_k$ and $L^p_k(\R^n)$
   be the subspaces of $H$, $V$ and $L^p(\R^n)$,
   respectively,
   which consist of functions satisfying the
   boundary condition \eqref{sys2a}.
   Then we have
   the embeddings
    $$V_k\subseteq H_k = H_k^*\subseteq V_k^*,
    \quad    L_k
    ^p(\R^n) \subseteq H_k =  H_k^*\subseteq 
    (L_k^p(\R^n))^*.
    $$
    By the monotone argument
    (see, e.g., \cite{rwan1}),
    we  find  that for every $k\in \N$,
    system \eqref{sys1a}-\eqref{sys3a}
    has a unique solution
           $u^\eps_k$  defined on $\mathcal{O}_k$
           which 
     is   a continuous $H_k$-valued
$\mathcal{F}_t$-adapted  process  with the property:
$$
u^\eps_k
 \in L^2(\Omega, C([0,T], H_k ))
 \bigcap L^2(\Omega, L^2(0,T;   V_k  )
 \bigcap L^p(\Omega, L^p(0, T;  L_k^p(\R^n)   ).
$$
Furthermore, 
for all $\xi\in V_k \bigcap L^p_k(\R^n)$,
$u^\eps_k$ satisfies the identity \eqref{defnsol_2}
with $u^\eps$ replaced by $u^\eps_k$.
Then by taking the limit of the sequence
$\{u^\eps_k\}_{k=1}^\infty$ in appropriate spaces
  as in \cite{rwan1},
we  can obtain a solution $u^\eps$ for \eqref{intr1}-\eqref{intr3}
in the sense of Definition \ref{defnsol}.
The uniqueness of solutions of 
\eqref{intr1}-\eqref{intr3} also follows from the argument
of
\cite{rwan1}, and the details   are omitted
here.

Next, we derive the uniform
estimates for the solution $u^\eps$.
By It\^{o}'s formula, we get from \eqref{intr1}-\eqref{intr3}
that for all $t\ge 0$, $P$-almost surely,
 $$
       \| {u}^{\eps}  (t) \|^2
       + 2\int_0^t \| (-\Delta)^{\frac \alpha{2}} {u}^{\eps}  (s)\|^2
       ds
       + 2\int_0^t \int_{\R^n}
        F (s, x, u^{\eps}(s) )   {u}^{\eps}  (s)    dx ds
       $$
       \be\label{energy1}
       = \|u_{0} \|^2 
                + 2 \int_0^t (u^{\eps} (s) , g (s) )ds
        +
       2\sqrt{\eps}
       \int_0^t 
     ( {u}^{\eps} (s) , 
        \sigma (s, u^{\eps}(s) )  dW )
       + \eps  \int_0^t \|  \sigma (s, u^{\eps}(s) )  \|^2_{\call_2(l^2,H)}
       ds.
       \ee
        By \eqref{sig6} we have
        for all $\eps\in (0,1)$,
       \be
       \label{le32 p2}
        \eps  \int_0^t \|  \sigma (s, u^{\eps}(s) )  \|^2_{\call_2(l^2,H)}
        \le
        {\frac 12}\lambda_1
        \int_0^t
        \int_{\R^n}
        |u^\eps (s,x)|^pdxds
        + 2\int_0^t
        \|\sigma_1 (s) \|^2_{L^2(\R^n,
        l^2)} ds
        + c_1t,
        \ee
        where $c_1>0$ is a constant
        depending only on
        $\lambda_1$, 
        $p,q, \kappa$, $\sum_{k=1}^\infty \beta_k$
        and $\sum_{k=1}^\infty \gamma_k$.
        It follows from 
        \eqref{f1} and \eqref{energy1}-\eqref{le32 p2}
        that for all $t\in [0,T]$,
        $$
       \| {u}^{\eps}  (t) \|^2
       + 2\int_0^t \| (-\Delta)^{\frac \alpha{2}} {u}^{\eps}  (s)\|^2
       ds
       + {\frac 32}\lambda_1
       \int_0^t \| u^\eps (s) \|^p_{L^p(\R^n)} ds
       $$
          $$
          \le
           \|u_{0} \|^2 
           +\int_0^t \|u^\eps (s)\|^2 ds
           +\|g\|^2_{L^2(0,T; H)}
           +    2  
        \|\sigma_1  \|^2_{L^2(0,T;   L^2(\R^n,
        l^2))}
        $$
        $$
        + c_1T
        +2 \|\psi_1\|_{L^1(0,T; L^1(\R^n))}
           +
       2\sqrt{\eps}
       \int_0^t 
     ( {u}^{\eps} (s) , 
        \sigma (s, u^{\eps}(s) )  dW ),
        $$
     and hence for all $0 \le t\le T$,
          $$
          \E
          \left (
          \sup_{0\le r \le t}
          \left (
       \| {u}^{\eps}  (r) \|^2
       + 2\int_0^r\| (-\Delta)^{\frac \alpha{2}} {u}^{\eps}  (s)\|^2
       ds
       + {\frac 32}\lambda_1
       \int_0^r\| u^\eps (s) \|^p_{L^p(\R^n)} ds
       \right )\right )
       $$
          $$
          \le  \E
          \left (
           \|u_{0} \|^2 
           \right )
           +\int_0^t
            \E
          \left ( \|u^\eps (s)\|^2 \right ) ds
           +\|g\|^2_{L^2(0,T; H)}
           +   
            2  
        \|\sigma_1  \|^2_{L^2(0,T;   L^2(\R^n,
        l^2))}
        $$
       \be\label{le32 p3}
        + c_1T
        +2 \|\psi_1\|_{L^1(0,T; L^1(\R^n))}
           +
       2\sqrt{\eps}  \E
          \left (
       \sup_{0\le r\le t} \left |
       \int_0^r 
     ( {u}^{\eps} (s) , 
        \sigma (s, u^{\eps}(s) )  dW )
        \right | \right ).
     \ee
     For the last term in \eqref{le32 p3},
     by the Burkholder inequality and \eqref{sig6} we obtain
     for all $\eps\in (0,1)$,
     $$
     2\sqrt{\eps} \E
          \left (
       \sup_{0\le r\le t} \left |
       \int_0^r 
     ( {u}^{\eps} (s) , 
        \sigma (s, u^{\eps}(s) )  dW )
        \right | \right )
        \le 6
         \E
          \left (
          \int_0^t
          \| \sigma (s, u^\eps (s))\|^2_{\call_2(l^2, H)}
          \|u^\eps (s) \|^2 ds
          \right )^{\frac 12}
        $$
        $$
        \le
         6
         \E
          \left (\sup_{0\le s\le t}  \|u^\eps (s) \| 
          \left (
          \int_0^t
          \| \sigma (s, u^\eps (s))\|^2_{\call_2(l^2, H)}
           ds
          \right )^{\frac 12}\right )
        $$
         $$
        \le
         {\frac 14}
         \E
          \left (\sup_{0\le s\le t}  \|u^\eps (s) \| ^2
          \right )
          +  36 \E  
          \left (
          \int_0^t
          \| \sigma (s, u^\eps (s))\|^2_{\call_2(l^2, H)}
           ds \right )
        $$
          \be\label{le32 p4}
        \le
         {\frac 14}
         \E
          \left (\sup_{0\le s\le t}  \|u^\eps (s) \| ^2
          \right )
          +  {\frac 14}\lambda_1
          \E \left (
          \int_0^t \|u^\eps (s)\|^p_{L^p(\R^n)} ds
          \right )
          + 72 
        \|\sigma_1  \|^2_{L^2(0,T;   L^2(\R^n,
        l^2))}
          + c_2 T,
      \ee
     where $c_2>0$ is a constant
        depending only on
        $\lambda_1$, 
        $p,q, \kappa$, $\sum_{k=1}^\infty \beta_k$
        and $\sum_{k=1}^\infty \gamma_k$.
        
      By \eqref{le32 p3}-\eqref{le32 p4} we get 
         for all $t\in [0,T]$,
        $$
          \E
          \left (
          \sup_{0\le r \le t}
          \left (
       \| {u}^{\eps}  (r) \|^2
       + 2\int_0^r\| (-\Delta)^{\frac \alpha{2}} {u}^{\eps}  (s)\|^2
       ds
       + {\frac 32}\lambda_1
       \int_0^r\| u^\eps (s) \|^p_{L^p(\R^n)} ds
       \right )\right )
       $$
          $$
          \le  \E
          \left (
           \|u_{0} \|^2 
           \right )
           +\int_0^t
            \E
          \left ( \|u^\eps (s)\|^2 \right ) ds
           +\|g\|^2_{L^2(0,T; H)}
         $$
       \be\label{le32 p5}
       +
        {\frac 14}
         \E
          \left (\sup_{0\le s\le t}  \|u^\eps (s) \| ^2
          \right )
          +  {\frac 14}\lambda_1
          \E \left (
          \int_0^t \|u^\eps (s)\|^p_{L^p(\R^n)} ds
          \right )
        + c_3,
      \ee
      where $c_3>0$ depends on $T$, but not on $u_0$
      or $\eps$.
      By \eqref{le32 p5} we obtain
        for all $t\in [0,T]$,
        $$
          \E
          \left (
          \sup_{0\le r \le t}
       \| {u}^{\eps}  (r) \|^2
       \right )
       + 2 \E
          \left (
          \int_0^t\| (-\Delta)^{\frac \alpha{2}} {u}^{\eps}  (s)\|^2
       ds \right )
       + {\frac 32}\lambda_1
        \E
          \left (
       \int_0^t \| u^\eps (s) \|^p_{L^p(\R^n)} ds
       \right ) 
       $$
          $$
          \le   3 \E
          \left (
           \|u_{0} \|^2 
           \right )
           + 3 \int_0^t
            \E
          \left ( \|u^\eps (s)\|^2 \right ) ds
           +3\|g\|^2_{L^2(0,T; H)}
          $$
       $$
       +
        {\frac 34}
         \E
          \left (\sup_{0\le s\le t}  \|u^\eps (s) \| ^2
          \right )
          +  {\frac 34}\lambda_1
          \E \left (
          \int_0^t \|u^\eps (s)\|^p_{L^p(\R^n)} ds
          \right )
        + 3 c_3,
     $$
      and
       hence 
        for all $t\in [0,T]$,
        $$ 
          \E
          \left (
          \sup_{0\le r \le t}
       \| {u}^{\eps}  (r) \|^2
       \right )
       + 8 \E
          \left (
          \int_0^t\| (-\Delta)^{\frac \alpha{2}} {u}^{\eps}  (s)\|^2
       ds \right )
       + 3\lambda_1
        \E
          \left (
       \int_0^t \| u^\eps (s) \|^p_{L^p(\R^n)} ds
       \right ) 
       $$
          $$
          \le   12 \E
          \left (
           \|u_{0} \|^2 
           \right )
           + 12 \int_0^t
            \E
          \left ( \sup_{0\le r\le s} \|u^\eps (r)\|^2 \right ) ds
           + 12\|g\|^2_{L^2(0,T; H)}
        + 12 c_3,
     $$
     which along with Gronwall's inequality
     implies that for all $t\in [0,T]$,
      $$ 
          \E
          \left (
          \sup_{0\le r \le t}
       \| {u}^{\eps}  (r) \|^2
       \right )
       + 8 \E
          \left (
          \int_0^t\| (-\Delta)^{\frac \alpha{2}} {u}^{\eps}  (s)\|^2
       ds \right )
       + 3\lambda_1
        \E
          \left (
       \int_0^t \| u^\eps (s) \|^p_{L^p(\R^n)} ds
       \right ) 
       $$
          $$
          \le  \left (  12 \E
          \left (
           \|u_{0} \|^2 
           \right )  
           + 12 \|g\|^2_{L^2(0,T; H)}
        + 12 c_3
        \right ) e^{12 t}.
     $$
     This completes the proof.
  \end{proof}

      In terms of Lemma \ref{le32}, 
      we know  that for every
       $\eps\in (0,1)$ and
      $u_0\in H$, there exists
      a 
       Borel-measurable map
     $\calg^\eps_{u_0}: C([0,T], U) \to  C([0, T], H )\bigcap  L^2(0, T;  V  )
     \bigcap  L^p(0, T;  L^p(\R^n)  )$
     such that the solution $u^\eps
     (\cdot, u_0) $ of
     \eqref{intr1}-\eqref{intr3} 
     defined on $[0,T]$
     with initial value $u_0$
     can be written as:
     $$
     u^\eps(\cdot, u_0) =\calg^\eps_{u_0} (W),
     \quad  \text{P-almost surely}.
     $$
     In the next section, we
     will
     study the   uniform LDP of
     the family
     $\{\calg^\eps_{u_0}\}$ 
     as $\eps \to 0$
     in the space
     $C([0, T], H )\bigcap  L^2(0, T;  V  )
     \bigcap  L^p(0, T;  L^p(\R^n)  )$.

   \section{Uniform LDP   of  reaction-diffusion
    equations} 
\setcounter{equation}{0}

In this section, we  first prove the
Freidlin-Wentzell
uniform LDP 
 of solutions to  \eqref{intr1}-\eqref{intr3}
 over bounded initial data,
 and then prove the 
 Dembo-Zeitouni uniform LDP
 over compact initial data.
  The
Freidlin-Wentzell
uniform LDP 
 of   \eqref{intr1}-\eqref{intr3}
 is stated below.

 \begin{thm}\label{main}
 	If \eqref{f1}-\eqref{f3}
 	and 
 	\eqref{sig1}-\eqref{sig4} hold true, then the solutions
 	of   
       \eqref{intr1}-\eqref{intr3}
       satisfy the LDP in
       $C([0, T], H )\bigcap  L^2(0, T;  V  )
       \bigcap  L^p(0, T;  L^p(\R^n)  )$  as $\eps \to 0$,
       which is uniform
       over  bounded
       initial data in $H$
       with   good rate function
 given by \eqref{rate}.
\end{thm}

 In order to prove 
 Theorem \ref{main}, we consider the
 the controlled equation
 with control     $v\in L^2(0,T; l^2)$:
   \be\label{contr1}
{\frac {d u_v (t)}{dt}}
+  (-\Delta)^\alpha u_v  (t)  
+F(t, \cdot, u_v (t) )=  
  g (t)  
 +   \sigma(t, u_v (t))  
    v (t),
\ee
 with  initial data
 \be\label{contr2}
 u_v  (0)=u_0 \in H.
 \ee
 Given 
   $v\in L^2(0,T; l^2)$
   and
   $u_0\in H$, by a solution
   $u_v$ of \eqref{contr1}-\eqref{contr2}
   on $[0,T]$, we 
   mean
   $$u_v\in C([0, T], H )\bigcap  L^2(0, T;  V  )
   \bigcap  L^p(0, T;  L^p(\R^n)  )$$
   and for all $t\in [0,T]$ 
  \be\label{contr3}
   u_v (t)
   +\int_0^t
   (-\Delta)^\alpha u_v(s) ds
   +\int_0^t
   F(s,\cdot, u_v)ds
   =u_0 +\int_0^t 
   g(s) ds +\int_0^t
   \sigma(s, u_v(s)) v(s) ds,
\ee
   in $\left (
   V\bigcap L^p(\R^n) \right )^*$.
   
   For the existence and uniqueness
   of solutions to
   \eqref{contr1}-\eqref{contr2},
   we have  the following result.

 \begin{lem}\label{cosol}
If    \eqref{f1}-\eqref{f3}
  and \eqref{sig1}-\eqref{sig4}  hold,
  then for 
   every 
  $v\in L^2(0,T; l^2)$ and $u_0\in H$,
    problem
\eqref{contr1}-\eqref{contr2}
 has a unique solution
 $u_v(\cdot, u_0)$ in
  $ C([0,T], H)\bigcap L^2(0,T;V)
 \bigcap L^p(0,T; L^p(\R^n))$.

 In addition, 
     for every $R>0$,
    there  exists
       $L_2 =L_2  (R, T)>0$ such that  
       for  all $u_{0,1}, u_{0,2}
       \in  \overline{B}_R(H)$
        and  
    $v_1, v_2\in \overline{B}_R ( L^2(0, T; l^2))$, 
    the solutions 
     $u_{v_1}(\cdot, u_{0,1})$ and $u_{v_2}(\cdot,
     u_{0,2}) $
    of  \eqref{contr1}-\eqref{contr2}   
    satisfy  for all $t\in [0,T]$,
    $$
    \| u _{v_1} (t, u_{0,1} ) -u _{v_2} (t, u_{0,2} ) \| ^2
    +   \int_0^T \| u_{v_1} 
    (t, u_{0,1} ) -u_{v_2} (t, u_{0,2}) \|^2_{V}   dt
    $$
    $$
    +
    \int_0^T \| u_{v_1} 
    (t, u_{0,1} ) -u_{v_2} (t, u_{0,2}) \|^p_{L^p(\R^n) }   dt
    $$
    \be\label{cosol 1}
    \le
   L_2
    \left (
    \| u_{0,1}- u_{0,2}\|^2
    +\|  v_1- v_2 \|^2_{L^2(0, T; l^2)}
    \right ),
   \ee
   and
    \be\label{cosol 2}
    \| u_{v_1} (t, u_{0,1} ) \| ^2
    +
      \int_0^T \|  u_{v_1} (t, u_{0,1} ) \|^2_{V}
   dt 
   + 
   \int_0^T
   \| u_{v_1} (t, u_{0,1} )\|^p_{L^p(\R^n)} dt
    \le
   L_2 .
   \ee
 \end{lem}

  \begin{proof}
  The existence of solutions of 
  \eqref{contr1}-\eqref{contr2}
  can be obtained by the Galerkin method
  for every 
   $v\in L^2(0,T; l^2)$ 
   and $u_0\in H$ when the
   conditions
    \eqref{f1}-\eqref{f3}
  and \eqref{sig1}-\eqref{sig4} are fulfilled.
  In what follows, we  prove \eqref{cosol 1}
  and \eqref{cosol 2}  from which the
  uniqueness of solutions follows.

     Let $u_1(t) = 
     u_{v_1} (t, u_{0,1})$
     and
     $u_2(t) = 
     u_{v_2} (t, u_{0,2})$
   for $t\in [0,T]$. 
    It follows from   \eqref{contr1},
   that 
    $$
   {\frac 12} {\frac d{dt}}
   \| u_1(t) \|^2
   +\| (-\Delta)^{\frac {\alpha}2} u_1(t)\|^2
    +\int_{\R^n} F(t,x, u_1(t))  u_1(t)  dx 
   $$
   \be\label{cosol p1}
   =
     (g(t), u_1(t))
   +(\sigma (t, u_1(t)) v_1(t), u_1(t)).
  \ee
  By \eqref{f1} and  \eqref{cosol p1}
  we obtain
  $$
    {\frac d{dt}}
   \| u_1(t) \|^2
   + 2 \| (-\Delta)^{\frac {\alpha}2} u_1(t)\|^2
    + 2\lambda_1
    \| u_1(t) \|^p_{L^p(\R^n)} 
   $$
   \be\label{cosol p2}
   \le  \| u_1(t) \|^2 + \| g(t) \|^2
   + 2 \|\psi_1 (t) \|_{L^1(\R^n)}
    + 2 (\sigma (t, u_1(t)) v_1(t), u_1(t)).
  \ee
   By \eqref{sig6} we have
   for  all   $t\in [0,T]$,
 $$
   2| (\sigma (t, u_1(t)) v_1(t), u_1(t))| 
  \le   2 \|  \sigma (t, u_1(t)) \|_{\call_2(l^2, H)}
    \|v_1(t) \|_{l^2}  \| u_1(t) \|
    $$
    $$
   \le 
   \| \sigma (t, u_1(t))\|^2_{\call_2(l^2, H)} 
  + 
  \| v_1(t)\|^2_{l^2}\| u_1(t) \|^2
  $$
 \be\label{cosol p3}
  \le  \lambda_1 \| u_1(t)\|^p_{L^p
  (\R^n)}
  + 2\|\sigma_1
  (t) \|^2_{L^2(\R^n, l^2)}
  +  \| v_1(t)\|^2_{l^2}
  \| u_1(t) \|^2 +c_1,
\ee
  where $c_1>0$ depends
  only on $\lambda_1, p, q,  \kappa,
  \sum_{k=1}^\infty \beta_k$
  and 
 $\sum_{k=1}^\infty \beta_k$.
 
 It follows from \eqref{cosol p2}-\eqref{cosol p3}
 that for all $t\in [0, T]$,
 $$
    {\frac d{dt}}
   \| u_1(t) \|^2
   + 2 \| (-\Delta)^{\frac {\alpha}2} u_1(t)\|^2
    +  \lambda_1
    \| u_1(t) \|^p_{L^p(\R^n)} 
   $$
   \be\label{cosol p4}
   \le  (1+ \|v_1(t)\|^2_{l^2}
   )  \| u_1(t) \|^2 + \| g(t) \|^2
   + 2\|\sigma_1
  (t) \|^2_{L^2(\R^n, l^2)}
   + 2 \|\psi_1 (t) \|_{L^1(\R^n)} 
   + 
   c_1.
  \ee 
 By \eqref{cosol p4}
 and Gronwall's inequality  we  get that
  for all $t\in [0, T]$,
  $\|u_{0,1}\|\le R$ and $\|v_1\|_{L^2
  (0,T; l^2)} \le R$,
  $$
   \| u_1(t) \|^2 
   + 2  \int_0^t \| (-\Delta)^{\frac {\alpha}2} u_1(s)\|^2 ds
    +  \lambda_1
   \int_0^t  \| u_1(s) \|^p_{L^p(\R^n)}  ds
   $$
   $$
   \le 
    e^{   \int_0^t  
   \left (
   1
    +    \| v_1(r)\|^2_{l^2}
   \right )
      dr }\|u_{0,1} \|^2
   + 
    \int_0^t
    e^{
        \int_s^t
   \left (
  1 
     +    \| v_1(r)\|^2_{l^2}
   \right )
     dr 
      } 
        ( 
       \| g(s) \|^2
       + 2\|\sigma_1
  (s) \|^2_{L^2(\R^n, l^2)}
      + 2 \| \psi_1 (s) \|_{L^1(\R^n)}
      +c_1
        )
      ds
   $$
   $$
     \le 
     e^{T  +   \int_0^T     \| v_1(r)\|^2_{l^2}
      dr }\|u_{0,1}\|^2 
     +
      e^{T  +   \int_0^T     \| v_1(r)\|^2_{l^2}
      dr }
      \int_0^T
        ( c_1
      +\| g(s) \|^2
      + 2\|\sigma_1
  (s) \|^2_{L^2(\R^n, l^2)}
      + 2 \| \psi_1 (s) \|_{L^1(\R^n)}
       )
      ds
      $$
\be\label{cosol p5}
     \le R^2
     e^{T  +   R^2}
     +
      e^{T  +   R^2}
        \left ( c_1T
      +\| g\|^2_{L^2(0,T; H)}
      + 2\|\sigma_1
   \|^2_{L^2(0,T; L^2(\R^n, l^2))}
      + 2 \| \psi_1   \|_{L^1(0,T; L^1(\R^n))}
      \right ),
\ee
  which implies \eqref{cosol 2}.

    We now  prove  \eqref{cosol 1}.
   By \eqref{cosol p5}, we find that
  there exists $c_3=c_3(R, T)>0$
  such that for all
    $\| u_{0,1} \|\le R$,  $\| u_{0,2} \|\le R$, 
    $\| v_1\|_{L^2(0, T; l^2)}\le R$
    and $\| v_2\|_{L^2(0, T; l^2)}\le R$, 
    \be
    \label{cosol p7}
      \| u_1  \|_{C([0,T], H)} +   \| u_2  \|_{C([0,T], H)}
      +\int_0^T \left (
      \| u_1 (t) \|^p_{L^p(\R^n)}
      +
       \| u_2 (t) \|^p_{L^p(\R^n)}
       \right ) dt
   \le   c_3 .
    \ee 
      By \eqref{contr1}-\eqref{contr2}
     we  have
     $$
     {\frac d{dt}}
     \| u_1(t) -u_2(t)\|^2
     +2  \| (-\Delta)^{\frac {\alpha}2}
     (u_1(t)-u_2(t))\|^2
     $$
     $$
    =-
     2\int_{\R^n}
      (F(t,x, u_1(t)) -F(t,x, u_2(t)))  (u_1(t)-u_2(t)) dx
     $$
  $$
     +2
     \left (
     \sigma (t, u_1(t))v_1(t)-
     \sigma (t, u_2(t))v_2(t),
     \  u_1(t)-u_2(t)
     \right ),
    $$
    which along with \eqref{f3} yields
     $$
     {\frac d{dt}}
     \| u_1(t) -u_2(t)\|^2
     +2  \| (-\Delta)^{\frac {\alpha}2}
     (u_1(t)-u_2(t))\|^2
     $$
     $$
     + 2\lambda_2
     \int_{\R^n}
     \left (
     |u_1(t,x)|^{p-2} u_1(t,x)
     -
       |u_2(t,x)|^{p-2} u_2(t,x)
       \right )(u_1(t,x) - u_2(t,x)  ) dx
     $$
    \be\label{cosol p8}
    \le   2\|\psi_4(t)\|_{L^\infty (\R^n)}
    \|u_1(t)-u_2(t) \|^2
     +2
     \left (
     \sigma (t, u_1(t))v_1(t)-
     \sigma (t, u_2(t))v_2(t),
     \  u_1(t)-u_2(t)
     \right ) .
     \ee
    Note that
    $$
   2
     \left (
     \sigma (t, u_1(t))v_1(t)-
     \sigma (t, u_2(t))v_2(t),
     \  u_1(t)-u_2(t)
     \right )
     $$ 
      $$
      =
   2
     \left (
     (\sigma (t, u_1(t))
     -
     \sigma (t, u_2(t)) )
     v_1(t)
     + \sigma (t, u_2(t)) (v_1 (t) -v_2 (t)), 
     \  u_1(t)-u_2(t)
     \right )
     $$ 
     $$
   \le 2
     \| 
    \sigma (t, u_1(t))-
     \sigma (t, u_2(t))\|_{\call_2( {l^2},H)}
     \|v_1(t)\| _{l^2}
     \| u_1(t)-u_2(t) \|
     $$ 
     $$
   + 2
     \|  
     \sigma (t, u_2(t))\|_{\call_2( {l^2},H)}
     \| v_1(t) - 
     v_2(t) \|_ {l^2}
     \| u_1(t)-u_2(t) \|
     $$
      $$
   \le  
     \| 
    \sigma (t, u_1(t))-
     \sigma (t, u_2(t))\|_{\call_2( {l^2},H)}^2
     +
     \|v_1(t)\| _{l^2}^2
     \| u_1(t)-u_2(t) \|^2
     $$ 
     \be\label{cosol p9}
   +  
     \|  
     \sigma (t, u_2(t))\|_{\call_2( {l^2},H)}^2
     \| u_1(t)-u_2(t) \|^2
     +
      \| v_1(t) - 
     v_2(t) \|_ {l^2}^2.
     \ee

     By  \eqref{elity2}  and \eqref{sig7}   we  obtain
     $$
      \| 
    \sigma (t, u_1(t))-
     \sigma (t, u_2(t))\|^2_{\call_2(l^2,H)}
     \le c_4    \| u_1(t)-u_2(t) \|^2
     $$
     $$
     + {\frac 12} \lambda_2
     \int_{\R^n} \left (
     |u_1(t, x)|^{p-2}
     +|u_2(t,x)|^{p-2} \right )
     \left (u_1(t,x)-u_2(t,x) \right )^2 dx
   $$
  \be\label{cosol p10}
   \le c_4    \| u_1(t)-u_2(t) \|^2
 +  \lambda_2
     \int_{\R^n} \left (
     |u_1|^{p-2}u_1 
     - |u_2|^{p-2}u_2
      \right )
     \left (u_1(t,x)-u_2(t,x) \right )  dx,
     \ee
   where $c_4>0$ depends only on
   $\lambda_2$, $p, q, \kappa$ and $\sum_{k=1}^\infty
   \alpha_k$.
   By \eqref{sig6}  we
       have
     \be\label{cosol p11} 
      \|  
     \sigma (t, u_2(t))\|^2_{\call_2(l^2,H)}
     \le  \|u_2 (t) \|^p_{L^p(\R^n)} 
     + c_5, 
     \ee 
     where $c_5>0$   depends only on
     $ \| \sigma_1 \|  _{C([0,T], L^2(\R^n, l^2))}$,
     $p, q, \kappa$,
     $\sum_{k=1}^\infty
   \beta_k$
      and $\sum_{k=1}^\infty
   \gamma_k$.

   By \eqref{cosol p9}-\eqref{cosol p11} we obtain
    $$
   2
     \left (
     \sigma (t, u_1(t))v_1(t)-
     \sigma (t, u_2(t))v_2(t),
     \  u_1(t)-u_2(t)
     \right )
     $$ 
     $$
     \le
    \lambda_2
     \int_{\R^n} \left (
     |u_1|^{p-2}u_1 
     - |u_2|^{p-2}u_2
      \right )
     \left (u_1(t,x)-u_2(t,x) \right )  dx
     $$
    \be\label{cosol p12}
   +
     \left (
     c_4 + c_5
     +\| v_1 (t) \|_{l^2}^2
     + \| u_2 (t) \|^p_{L^p(\R^n)}
     \right ) \| u_1 (t) -u_2 (t) \|^2
     + \| v_1 (t) -v_2 (t) \|^2.
     \ee
     By \eqref{cosol p8} and \eqref{cosol p12} we get
     for $t\in (0,T)$,
      $$
     {\frac d{dt}}
     \| u_1(t) -u_2(t)\|^2
     +2  \| (-\Delta)^{\frac {\alpha}2}
     (u_1(t)-u_2(t))\|^2
     $$
     $$
     +  \lambda_2
     \int_{\R^n}
     \left (
     |u_1(t,x)|^{p-2} u_1(t,x)
     -
       |u_2(t,x)|^{p-2} u_2(t,x)
       \right )(u_1(t,x) - u_2(t,x)  ) dx
     $$ 
    \be\label{cosol p13}
   +
     \left (
      c_6
     +\| v_1 (t) \|_{l^2}^2
     + \| u_2 (t) \|^p_{L^p(\R^n)}
     \right ) \| u_1 (t) -u_2 (t) \|^2
     + \| v_1 (t) -v_2 (t) \|^2,
     \ee
    where
    $c_6=  c_4 + c_5 +2\|\psi_4 \|_{
    L^\infty(0,T; L^\infty (\R^n))}$.
     
     By \eqref{cosol p13} and Gronwall's inequality, we get
     for all $t\in [0, T]$,
      $$ 
     \| u_1(t) -u_2(t)\|^2
     +2  \int_0^t \| (-\Delta)^{\frac {\alpha}2}
     (u_1(s)-u_2(s))\|^2 ds
     $$
     $$
     +  \lambda_2\int_0^t 
     \int_{\R^n}
     \left (
     |u_1(s,x)|^{p-2} u_1(s,x)
     -
       |u_2(s,x)|^{p-2} u_2(s,x)
       \right )(u_1(s,x) - u_2(s,x)  ) dx ds
     $$ 
   $$
      \le e^{
      \int_0^t 
      \left (
      c_6
     +\| v_1 (r) \|_{l^2}^2
     + \| u_2 (r) \|^p_{L^p(\R^n)}
     \right )
       dr
      } \| u_{0,1}  -u_{0,2}   \|^2
      $$
      $$
            +  
    \int_0^t
    e^{
      \int_s^t 
     \left (
      c_6
     +\| v_1 (r) \|_{l^2}^2
     + \| u_2 (r) \|^p_{L^p(\R^n)}
     \right )
      dr
      }
      \|v_1(s) - v_2 (s)  \|_{l^2}^2
      ds
      $$
      $$
      \le e^{
      \int_0^T 
      \left (
      c_6
     +\| v_1 (r) \|_{l^2}^2
     + \| u_2 (r) \|^p_{L^p(\R^n)}
     \right )
       dr
      } \| u_{0,1}  -u_{0,2}   \|^2
      $$
     \be\label{cosol p14}
            + e^{
      \int_0^T 
     \left (
      c_6
     +\| v_1 (r) \|_{l^2}^2
     + \| u_2 (r) \|^p_{L^p(\R^n)}
     \right )
      dr
      } 
    \int_0^t
      \|v_1(s) - v_2 (s)  \|_{l^2}^2
      ds.
      \ee
   It follows from \eqref{cosol p7}
   and \eqref{cosol p14} that  
   for any $\| v_1\|_{L^2(0,T; l^2)}
   \le R$ and $ t\in [0,T]$,
      $$ 
     \| u_1(t) -u_2(t)\|^2
     +2  \int_0^t \| (-\Delta)^{\frac {\alpha}2}
     (u_1(s)-u_2(s))\|^2 ds
     $$
     $$
     +  \lambda_2\int_0^t 
     \int_{\R^n}
     \left (
     |u_1(s,x)|^{p-2} u_1(s,x)
     -
       |u_2(s,x)|^{p-2} u_2(s,x)
       \right )(u_1(s,x) - u_2(s,x)  ) dx ds
     $$ 
    \be\label{cosol p15}
      \le c_7 \| u_{0,1}  -u_{0,2}   \|^2
      +c_7
           \|v_1  - v_2    \|^2_{L^2(0,T; l^2)},
    \ee
    where $c_7=c_7(R , T)>0$, which
    along with \eqref{elity1}  yields
    \eqref{cosol 1} and thus completes the proof.
  \end{proof}

 In the sequel, we will prove the continuity
 of the solution $u_v$ of the controlled equation
  \eqref{contr1}  in $v$ with  respect to
   the weak topology of  $L^2(0,T; l^2)$
   for which  we need the uniform tail-ends estimates
   of $u_v$
   as given below.

 \begin{lem}\label{tail}
 If    \eqref{f1}-\eqref{f3}
  and \eqref{sig1}-\eqref{sig4} are fulfilled,
  then for every
  $T>0$, $u_0\in H$,  
  $R>0$ and $\eps>0$,
  there exists $K=K(T, u_0, R, \eps)>0$ such that
  for all $v\in \overline{B}_R (L^2(0,T; l^2))$,
  the solution
  $u_v$ of \eqref{contr1}-\eqref{contr2}
  satisfies, for all $m\ge K$  and  $t\in [0,T]$,
   $$
   \int_{|x|\ge m}
   \left |u_v(t, x) \right |^2 dx
    +  \int_0^T \int_{|x| \ge m }
   |u(s,x)|^p dx ds
 < \eps.
 $$
   \end{lem}
     
 \begin{proof}
 Choose a smooth function
 $\theta: \R^n \to [0,1]$    such that 
 \be\label{cutoff}
 \theta (x)=0 \ \ \text{for }  |x| \le {\frac 12};    \  \ \
 \theta (x) = 1 \  \  
 \text{for } \  |x| \ge 1.
\ee
Denote by 
 $\theta_m (x)  =  \theta  \left (
 {\frac xm}
 \right ) $ for all $m\in \N$  and $x\in \R^n$.
  By \eqref{contr1}  we obtain
  \be\label{ta 1}
 {\frac {d}{dt}}
 \|\theta_m u_v (t) \|^2 
 + 2 
   (   (-\Delta)^ {\frac {\alpha}2}   u_v(t),  
    (-\Delta)^ {\frac {\alpha}2} (\theta_m^2 u_v (t) ))
  + 2  \int_{\R^n}
   F(t, x, u_v (t))   \theta_m^2 (x)   u_v (t) dx  
 $$
 $$
 = 2   (\theta_m g (t), \theta_m  u_v (t) ) 
  + 2(  \theta_m
  \sigma (t, u_v(t)) v(t),
  \theta_m  u_v(t))  .
 \ee
 Note that
  $$
 2 
    (   (-\Delta)^ {\frac {\alpha}2}   u_v(t),  
     (-\Delta)^ {\frac {\alpha}2} (\theta_m^2 u_v(t)) )
   $$
   $$
   =
   C(n,\alpha) 
   \int_{\R^n}\int_{\R^n} 
   {\frac
   { (u_v(t,x) -u_v(t,y)) (\theta_m^2(x) u_v(t,x)  
   -  \theta_m^2(y)  u_v(t,y)  ) }
   {|x-y|^{n+2\alpha}}
   }dxdy ,
  $$ 
  from which, after some calculations,  we can obtain
    \be\label{ta 2}
   -2 
    (   (-\Delta)^ {\frac {\alpha}2}   u_v(t),  
     (-\Delta)^ {\frac {\alpha}2} (\theta_m^2 u_v(t)) )
    \le c_1    m^{-\alpha}
   \| u_v (t) \|^2_V  ,
  \ee
  where $c_1>0$ is  
  independent of $m$.
  By \eqref{f1} and Young's inequality we obtain
  from \eqref{ta 1}-\eqref{ta 2} that
$$
 {\frac {d}{dt}}
 \|\theta_m u_v (t) \|^2 
 + 2 \lambda_1 \int_{\R^n}
 \theta_m ^2 (x) |u_v (t,x)|^p dx
 $$
 $$
 \le
 c_1 m^{-\alpha} \| u_v (t) \|_V^2
 +2\int_{\R^n} \theta_m ^2 (x) \psi_1 (t,x) dx
 +\|\theta_m g (t) \|^2 
 + \| \theta_m u_v (t) \|^2
 $$
   \be\label{ta 3}
   +   
  \| v(t)\|^2_{l^2} \| \theta_m  u_v(t) \|^2 
  +  \| \theta_m
  \sigma (t, u_v(t))\|_{\call_2(l^2, H)}^2
 .
 \ee
 We now deal with the last term in \eqref{ta 3}.
 By \eqref{sig3} and \eqref{sig5} we get
 $$
  \| \theta_m
  \sigma (t, u_v(t))\|_{\call_2(l^2, H)}^2
  =\sum_{k=1}^\infty
  \int_{\R^n}
  \theta^2_m (x) \left |
  \sigma_{1,k} (t,x)
  +\kappa (x)
  \sigma_{2,k} (t,x, u_v )
  \right |^2 dx
  $$
  $$
  \le
  2
  \sum_{k=1}^\infty
  \int_{\R^n}
  \theta^2_m (x) \left |
  \sigma_{1,k} (t,x) \right  |^2 dx
  +
  2\sum_{k=1}^\infty \beta_k
  \int_{\R^n}
  \theta^2_m (x)   
   \kappa^2  (x)dx
   $$
  \be\label{ta 4}
  +
  2 \sum_{k=1}^\infty \gamma_k
  \int_{\R^n}
  \theta^2_m (x)   
   \kappa^2  (x) | u_v(t,x)| ^q dx.
   \ee
   For the last term in \eqref{ta 4} we have
   $$2 \sum_{k=1}^\infty \gamma_k
  \int_{\R^n}
  \theta^2_m (x)   
   \kappa^2  (x) | u_v(t,x)| ^q dx
   $$
   $$
   \le
   2 \sum_{k=1}^\infty \gamma_k
   \left ( \int_{\R^n}
  \theta^2_m (x)   
   |\kappa  (x)|^{\frac {2p}{p-q}} dx
   \right )^{\frac {p-q}p}
    \left ( \int_{\R^n}
  \theta^2_m (x)   
   |u_v (t,  x)|^{p} dx
   \right )^{\frac {q}p}
   $$
   \be\label{ta 5}
   \le
   \lambda_1   \int_{\R^n}
  \theta^2_m (x)   
   |u_v (t,  x)|^{p} dx
   + c_2  
     \int_{\R^n}
  \theta^2_m (x)   
   |\kappa  (x)|^{\frac {2p}{p-q}} dx,
 \ee
   where
   $c_2>0$ depends only
   on $\lambda_1, p, q$ and
   $\sum_{k=1}^\infty \gamma_k$,
   but not on $m$.
   It follows from \eqref{ta 4}-\eqref{ta 5} that
   $$\| \theta_m
  \sigma (t, u_v(t))\|_{\call_2(l^2, H)}^2
  \le 
   \lambda_1   \int_{\R^n}
  \theta^2_m (x)   
   |u_v (t,  x)|^{p} dx
   $$
    \be\label{ta 6}
   +
  2
  \sum_{k=1}^\infty  
    \int_{|x|\ge {\frac 12}m}
   |\sigma_{1,k} (t,x)    |^2 dx
   + c_3
  \int_{|x|\ge {\frac 12}m}
  \left (\kappa^2 (x) 
  + |\kappa  (x)|^{\frac {2p}{p-q}} 
  \right ) dx,
   \ee
   where $c_3>0$ depends only on
   $c_2$  and $\sum_{k=1}^\infty \beta_k$.
   
   By \eqref{ta 3}
   and \eqref{ta 6} we obtain
   $$
 {\frac {d}{dt}}
 \|\theta_m u_v (t) \|^2 
 +   \lambda_1 \int_{\R^n}
 \theta_m ^2 (x) |u_v (t,x)|^p dx
 $$
 $$
 \le
 (1+ \| v(t)\|^2_{l^2})
  \| \theta_m  u_v(t) \|^2 
  +
 c_1 m^{-\alpha} \| u_v (t) \|_V^2
 +2
  \sum_{k=1}^\infty  
    \int_{|x|\ge {\frac 12}m}
   |\sigma_{1,k} (t,x)    |^2 dx
 $$
$$
   + 
  \int_{|x|\ge {\frac 12}m}
  \left (2|\psi_1 (t,x)|
  + g^2(t,x) + c_3\kappa^2 (x) 
  + c_3|\kappa  (x)|^{\frac {2p}{p-q}} 
  \right ) dx,
  $$ 
  which implies that for all $t\in [0,T]$
  and $v\in \overline{B}
  _R(L^2(l^2, H))$,
   $$
   \|\theta_m u_v (t) \|^2 
   +
   \lambda_1 \int_0^t  \int_{\R^n}
 \theta_m ^2 (x) |u_v (s,x)|^p dx ds
 $$
 $$
   \le
   e^{ 
   \int_0^t  (1+\| v(r)\|^2_{l^2})dr
   }\| \theta_m u_0\|^2
   +
 c_1 m^{-\alpha}  \int_0^t
  e^{ 
   \int_s^t  (1+\| v(r)\|^2_{l^2})dr
   }
 \| u_v (s)\|^2_V ds
   $$
   $$ 
 +
 2
  \int_0^t
  e^{ 
   \int_s^t  (1+\| v(r)\|^2_{l^2})dr
   }  \left (
 \int_{|x| \ge {\frac 12} m}\sum_{k=1}^\infty
  | \sigma_{1,k} (s,x ) |^2 dx
  \right )  ds
  $$
  $$
  +  
  \int_0^t
  e^{ 
   \int_s^t  (1+\| v(r)\|^2_{l^2})dr
   }
   \left (
   \int_{|x| \ge {\frac 12} m}
\left (
2|\psi_1 (s,x)|
  + g^2(s,x) + c_3\kappa^2 (x) 
  + c_3|\kappa  (x)|^{\frac {2p}{p-q}} 
\right )
   dx\right )  ds
 $$
  $$  
   \le
   e^{   T+  R^2 
   }\int_{|x| \ge {\frac 12 m}
   }  |u_0 (x)|^2 dx 
   +
 c_1   e^{   T+  R^2 
   }  m^{-\alpha}  \int_0^T 
 \| u_v (s)\|^2_V ds
   $$
  \be\label{ta 7}
  +  e^{   T+  R^2 
   }
  \int_0^T 
   \int_{|x| \ge {\frac 12} m}
\left (2| \sigma_{1} (s,x ) |^2_{l^2}
+
2|\psi_1 (s,x)|
  + g^2(s,x) + c_3\kappa^2 (x) 
  + c_3|\kappa  (x)|^{\frac {2p}{p-q}} 
\right )
   dx  ds.
   \ee

  Since 
  $\sigma_1 \in L^2(0,T; L^2(\R^n, l^2))$,
  $\psi_1\in L^1(0,T; L^1(\R^n))$,
  $g\in L^2(0,T, H)$ and
  $k\in L^2(\R^n)
  \bigcap L^\infty (\R^n)$, we infer that
  \be\label{ta 8}
   \lim_{m\to \infty}
   \int_0^T 
   \int_{|x| \ge {\frac 12} m}
\left (2| \sigma_{1} (s,x ) |^2_{l^2}
+
2|\psi_1 (s,x)|
  + g^2(s,x) + c_3\kappa^2 (x) 
  + c_3|\kappa  (x)|^{\frac {2p}{p-q}} 
\right )
   dx  ds =0.
   \ee
   
    On the other hand, 
    by Lemma \ref{cosol} we know that
    there exists $c_2=c_2 (R,T)>0 $ such  that
\be\label{ta 9}
   \lim_{m\to \infty}
   m^{-\alpha}  \int_0^T 
 \| u_v (s)\|^2_V ds
 \le    \lim_{m\to \infty}
  c_2    m^{-\alpha}  =0   .
\ee
By \eqref{ta 7}-\eqref{ta 9} we infer that
for every $T>0$, $R>0$, $u_0\in H$  and $\eps>0$,
there exists $m_1=m_1(T, u_0,  R,\eps)\in \N$ such that
for all $m\ge m_1$ and $t\in [0, T]$,
$$
   \|\theta_m u_v (t) \|^2 
   +
   \lambda_1 \int_0^t  \int_{\R^n}
 \theta_m ^2 (x) |u_v (s,x)|^p dx ds
 \le \eps,
 $$ 
  which together with the fact that
  $\theta_m (x) =1$ for all $|x| \ge 1$
  completes  the proof.
   \end{proof}

 Next, we  prove  
  the 
    continuity of
    the solution 
    $u_v$ of   \eqref{contr1}-\eqref{contr2}
     with respect to the control 
  $v$ in the weak topology of  $L^2(0,T; l^2)$.

 \begin{lem}\label{wc}
 Let  \eqref{f1}-\eqref{f3}
  and \eqref{sig1}-\eqref{sig4}  be fulfilled.
  If       $v_n \to v$ weakly in 
  $L^2(0,T; l^2)$,
  then
  $ u_{v_n} \to u_v $ strongly
  in  
 $  C([0,T], H)\bigcap L^2(0,T;V)
 \bigcap L^p(0,T; L^p(\R^n))
 $. 
   \end{lem}
 
 \begin{proof}
 We first derive the weak convergence  
 of the sequence
 $\{u_{v_n}\}_{n=1}^\infty$ and then consider the
 strong convergence   as  
  $v_n \to v$ weakly in 
  $L^2(0,T; l^2)$, based on the uniform estimates of solutions.  
  
  {\bf Step 1: weak convergence 
   of $\{u_{v_n}\}_{n=1}^\infty$.}
  By the weak convergence of 
 $\{v_n\}_{n=1}^\infty $  in 
  $L^2(0,T; l^2)$,  for every $u_0 \in H$, it follows
  from Lemma \ref{cosol} that 
  \be\label{wc 1}
  \| u_{v_n}\|_{C([0,T], H)}
  + \| u_{v_n}\|_{L^2(0,T; V)  }
  + \| u_{v_n}\|_{L^p(0,T; L^p(\R^n) )  }
  + \| v_n \|_{L^2(0,T;l^2)  }
  \le c_1, \quad \forall \ n\in \N,
  \ee
  where   $c_1=c_1(T)>0$  depends only  on $T$.
  By \eqref{f2}  and \eqref{wc 1} we have
   \be\label{wc 2}
   \| F(\cdot, \cdot, u_{v_n} )\| 
   _{L^q(0,T; L^q(\R^n))}
  \le c_2,
  \ee
   where $c_2=c_2(T)>0$.

  By \eqref{wc 1}
  and \eqref{wc 2}
   we  find  that
  there exists
  $\tilde{u} \in 
  L^\infty(0,T; H)\bigcap
  {L^2(0,T; V)  }\bigcap
  {L^p(0,T; L^p(\R^n) )  }$,
  $\chi \in L^q(0,T; L^q(\R^n))$
  such that,  up to a subsequence, 
  \be\label{wc 7}
  u_{v_n}
  \to  \tilde{u}
  \ \text{ weak-star  in  }  \  L^\infty(0,T; H),
  \ee
   \be\label{wc 8}
  u_{v_n}
  \to  \tilde{u}
  \ \text{ weakly  in  }  \  L^2(0,T; V),
  \ee 
   \be\label{wc 9}
  u_{v_n}
  \to  \tilde{u}
  \ \text{ weakly  in  }  \   L^p(0,T; L^p(\R^n) ) ,
  \ee
  and
  \be\label{wc 10}
  F(\cdot,\cdot, u_{v_n})
  \to \chi
  \ \text{ weakly  in  }  \   L^q(0,T; L^q(\R^n) ) .
  \ee
  
  Next, we prove a strong 
  convergence of $\{u_{v_n}\}_{n=1}^\infty
  $.

  {\bf Step 2: strong  convergence
   of $\{u_{v_n}\}_{n=1}^\infty
  $ in 
  $C([0,T], 
   (V\bigcap L^p(\R^n)   )^*)$
  }.  We now prove:
   \be\label{wc 20}
     u_{v_n}
  \to \tilde{u}
  \ \text{ strongly  in  }  \  C([0,T], 
   (V\bigcap L^p(\R^n)   )^*),
  \ee 
  by the   Arzela-Ascoli theorem.
  We first
  prove 
  the sequence 
   $\{u_{v_n}\}_{n=1}^\infty$
   is equicontinuous  in 
   $ 
   (V\bigcap L^p(\R^n)   )^* $ on $[0,T]$, and then
   prove  
  $\{ u_{v_n} (t)\}_{n=1}^\infty$ is precompact
  in $ 
    (V\bigcap L^p(\R^n)   )^* $ for
    every fixed  $t\in [0,T]$.

    Note that the solution $u_{v_n}$ of \eqref{contr1}-\eqref{contr2}
   satisfies \eqref{contr3}
   in $\left (V\bigcap L^p(\R^n)\right )^*$
   with $v$ replaced by $v_n$.
   Therefore, we have 
 for all $0\le  s \le t \le T$,
  $$
 \| u_{v_n} (t)- u_{v_n} (s)\|
 _{\left (V\bigcap L^p(\R^n)\right )^*}
\le \int_s^t
\|  (-\Delta)^\alpha u_{v_n}(r) 
\|_{V^*}dr
$$
  \be\label{wc 23}
  +\int_s^t \|  F(r,\cdot, u_{v_n}(r)) \|_{L^q(\R^n)}dr
  +   \int_s^t \| g(r) \| dr
  +\int_s^t \| \sigma (r, u_{v_n}(r)) v_n(r)\| ds.
\ee
By \eqref{wc 1}-\eqref{wc 2} we  see  
$$
   \int_s^t
\|  (-\Delta)^\alpha u_{v_n}(r) 
\|_{V^*}dr
+
   \int_s^t \|  F(r,\cdot, u_{v_n}(r)) \|_{L^q(\R^n)}dr
  +   \int_s^t \| g(r) \| dr
  $$
     \be\label{wc  24}
  \le   
  ( c_1 + \| g \|_{L^2(0,T; H)} ) (t-s)^{\frac 12} 
 +   c_2 (t-s)^{\frac 1p}. 
  \ee
  By \eqref{sig5a}  and \eqref{wc 1} we obtain
  $$
 \int_s^t \| \sigma (r, u_{v_n}(r)) v_n(r)\| dr
 \le
   \left (
   \int_s^t \| \sigma (r, u_{v_n}(r))\|^2_{\call_2(l^2,H)}
   dr
   \right )^{\frac 12}
\|  v_n \|_{L^2(0,T; l^2 ) }   
 $$
 $$
  \le
  c_1 \left (
   \int_s^t 
   \ ( 
   2\|\sigma_1(r) \|^2_{L^2(\R^n, l^2)}
   + 2\sum_{k=1}^\infty \beta_k \|\kappa\|^2
   +2\sum_{k=1}^\infty \gamma_k
   \|\kappa\|^2_{L^{\frac {2p}{p-q}}
   (\R^n)}
   \|u_{v_n} (r)\|^q_{L^p(\R^n)}
    )dr
   \right )^{\frac 12} 
  $$
  $$
  \le c_1 c_3 (t-s)^{\frac 12}
  + c_1 c_4 
  \left (\int _s^t
  \|u_{v_n} (r)\|^q_{L^p(\R^n)}dr
  \right )^{\frac 12}
  $$
   $$
  \le c_1 c_3 (t-s)^{\frac 12}
  + c_1 c_4 \| u_{v_n}\|^{\frac q2}_{L^p(0,T;L^p(\R^n))}
  (t-s)^{\frac {p-q}{2p}} 
  $$
  \be\label{wc 25}
  \le c_1 c_3 (t-s)^{\frac 12}
  + c_1^{1+   {\frac q2}}  c_4   
  (t-s)^{\frac {p-q}{2p}} ,
  \ee 
  where
  $c_3
  =\left (
  2\|\sigma_1\|^2_{L^\infty(0,T;  L^2(\R^n, l^2))}
   + 2\sum_{k=1}^\infty \beta_k \|\kappa\|^2
   \right )^{\frac 12}$
   and $c_4=
   \left ( 
    2\sum_{k=1}^\infty \gamma_k \right )^{\frac 12}
   \|\kappa\| _{L^{\frac {2p}{p-q}}
   (\R^n)}
   $.
   By  \eqref{wc 23}-\eqref{wc 25}
  we obtain, 
  for all $0\le  s \le t \le T$,
  \be\label{wc 26}
 \| u_{v_n} (t)- u_{v_n} (s)\|
 _{\left (V\bigcap L^p(\R^n)\right )^*}
\le c_5 \left ( (t-s)^{\frac 12} +
(t-s)^{\frac 1p}
+ (t-s)^{\frac {p-q}{2p}}
\right ),
\ee
where $c_5=c_5(T)>0$.

  Since $p>q$, by \eqref{wc 26} we see that
  the sequence 
   $\{u_{v_n}\}_{n=1}^\infty$
   is equicontinuous  in 
   $ 
   (V\bigcap L^p(\R^n)   )^* $ on $[0,T]$.
   It remains to show  
  $\{ u_{v_n} (t)\}_{n=1}^\infty$ is precompact
  in $ 
    (V\bigcap L^p(\R^n)   )^* $ for
    every fixed  $t\in [0,T]$.

     Since  the Sobolev  embeddings are not compact
     on unbounded domains,  we need to 
     employ the uniform tail-ends estimates of solutions
     to prove the precompactness 
     of the sequence $\{ u_{v_n} (t)\}_{n=1}^\infty$  
  in $ 
    (V\bigcap L^p(\R^n)   )^* $
    as in \cite{wangJDE2019} where
    the    noise with linear growth was discussed.
    To that end, we
     decompose
     the solution
     $u_{v_n}$ as
     $$ u_{v_n} (t,x)=
     u^{1,m}  _{v_n} (t,x)
        +
          u^{2,m}  _{v_n} (t,x),
          \quad \forall \  t\in [0,T],\
          x\in \R^n,
          $$
          where $u^{1,m}$ and $u^{2,m}$
          are defined in terms of the cutoff function
          $\theta$ in \eqref{cutoff} as follows:
           \be\label{wc 30}
        u^{1,m}  _{v_n} (t,x)
        =\theta (
        {\frac {x}m}
         )  u _{v_n} (t,x) 
        \quad \text{and}
        \quad
          u^{2,m}  _{v_n} (t,x)
        =(1- \theta (
        {\frac {x}m}
         )  )  u _{v_n} (t,x), 
        \ee
         for all $m\in \N$, $t\in [0,T]$  and
        $x\in \R^n$.
        
        By
          \eqref{wc 1}
         and   the uniform estimates on the tails
         of solutions as obtained in Lemma \ref{tail},
           we  infer that
          for every $\eps>0$ and $u_0\in H$, there exists
          $m_1=m_1(T, u_0, \eps)>0$ such that
          for all   $t\in [0,T]$,
          $$
          \int_{|x|\ge {\frac 12} m_1}
          |u_{v_n} (t,x)|^2 dx <{\frac 18} \eps^2,
          \quad \forall \  n\in \N,
          $$
          and hence for all
        $ n\in \N$ and  $t\in [0,T]$,
          \be\label{wc 31}
       \| u^{1,m_1}_{v_n} (t )\|^2
         \le
       \int_{|x|\ge {\frac 12} m_1}
        |u_{v_n} (t,x)|^2 dx<{\frac 18} \eps^2.
   \ee
         Let 
         $ H_{m_1} =   \{u\in H:   
            u(x) =0  \
          \text{a.e. on }  |x|\ge m_1\} 
          $
          and
           $ V_{m_1} =   \{u\in V:   
            u(x) =0  \
          \text{a.e. on }  |x|\ge m_1\}
          $.
          Then
          the embedding 
          $H_{m_1}^* \hookrightarrow
          V_{m_1}^*$  is compact.
          Note that  for all $t\in [0,T]$ and $n\in \N$,
          by \eqref{wc 1} we have 
        $\|       u^{2,m_1}_{v_n } (t)\|_{ H_{m_1}} 
          \le c_1$,  and hence 
        the sequence  $\{ u^{2,m_1}
         _{v_n }(t)\}_{n=1}^\infty$
         is precompact in  
        $  (V\bigcap L^p(\R^n)   )^*$,
        which together with \eqref{wc 31}
        shows that  
    $  \{u_{v_n} (t)\}_{n=1}^\infty = 
          \{ u^{1,m_1}_{v_n} (t)
         +  u^{2,m_1}_{v_n} (t)\}
         _{n=1}^\infty$
         has a finite open cover with
         radius $\eps$ in  
         $(V\bigcap L^p(\R^n)   )^*$, and thus
         $  \{u_{v_n} (t)\}_{n=1}^\infty$ is
         precompact in  
         $(V\bigcap L^p(\R^n)   )^*$.
         As a result, \eqref{wc 20} follows from
        the Arzela-Ascoli theorem and \eqref{wc 7}.

In the next step, we further improve 
   the strong convergence of
  $\{u _{v_n}  \}_{n=1}^\infty$
  from the space \\
  $
  C([0,T], (V\bigcap L^p(\R^n)   )^*)$
 to the space $L^4(0,T; H)$.

  {\bf Step 3: strong  convergence
  of   $\{u _{v_n}  \}_{n=1}^\infty$ in
   $L^4(0,T; H)$.} We now prove
      \be\label{wc 40}
         u_{v_n}      \to  u_v
         \ \text{ strongly in } \
         L^4(0,T; H).
         \ee
     First we have    
          $$
         \int_0^T
         \| u_{v_n} (t)  - \tilde{u}  (t)\|^4 dt
         =\int_0^T
         (u_{v_n} (t)  - \tilde{u}  (t),\
         u_{v_n} (t)  - \tilde{u}  (t))^2_{
          (V\bigcap L^p(\R^n),
          (V\bigcap L^p(\R^n)   )^* ) 
         }
         $$
         $$
         \le  
          \| u_{v_n}   - \tilde{u} \|_{C([0,T],
         (V\bigcap L^p(\R^n)   )^*)  }^2  
         \int_0^T 
         \| u_{v_n} (t)  - \tilde{u} (t)\|_{
        V\bigcap L^p(\R^n)  }^2 dt
       $$ 
       $$
         \le   2
         \left (
         \| u_{v_n}   \|_{L^2(0,T;
        V\bigcap L^p(\R^n) )}^2
        +
        \|\tilde{u}   \|_{L^2(0,T;
        V\bigcap L^p(\R^n) )}^2
         \right )
          \| u_{v_n}   - \tilde{u} \|_{C([0,T],
         (V\bigcap L^p(\R^n)   )^*)  }^2  .
        $$ 
        Then by  
          \eqref{wc 1}
         and \eqref{wc 20} we obtain 
        \be\label{wc 41}
         u_{v_n}      \to \tilde{u}
         \ \text{ strongly in } \
         L^4(0,T; H).
         \ee    
         
         It remains to show 
           $\tilde{u} =u_v$; that is,
           $\tilde{u}$ is the solution of
           \eqref{contr1}-\eqref{contr2}.
           Since $u_{v_n}$ is the solution
           of
                      \eqref{contr1}-\eqref{contr2}
                      with $v$ replaced by $v_n$,
                      by 
         \eqref{contr3} we see that
       for  $0 \le t\le T$
       and $\xi \in  V\bigcap L^p(\R^n) $,
   \be\label{wc 42}
  (u_{v_n} (t), \xi)
  +\int_0^t 
  ( (-\Delta)^{\frac {\alpha}2}  u_{v_n}(s),
  (-\Delta)^{\frac {\alpha}2}  \xi
  ) ds
  +\int_0^t \int_{\R^n} F(s,x, u_{v_n}(s))
  \xi (x) dx  ds
  $$
  $$
  =(u_0,\xi)  +\int_0^t (g(s),\xi)  ds
  +\int_0^t 
  (\sigma (s, u_{v_n}(s)) v_n(s), \xi)  ds.
\ee 
We  will take the
limit of 
 \eqref{wc 42} to show 
  $\tilde{u} =u_v$, by starting
   with the last term in \eqref{wc 42}
  which has a superlinear growth in $u_{v_n}$.
    We claim
  that for every
       $t\in [0,T]$,
       \be\label{wc 43}
      \lim_{n\to \infty}
       \int_0^t
       \left ( \sigma (s, u_{v_n}(s))-
        \sigma (s, \tilde{u}  (s)) \right )
         v_n(s)   ds  =0
         \ \text{ in } \ H.
         \ee
    By \eqref{sig9}
    and \eqref{wc 1} we  get
     $$
     \|  \int_0^t
       \left ( \sigma (s, u_{v_n}(s))-
        \sigma (s, \tilde{u}  (s)) \right )
         v_n(s)   ds\|
         \le
          \int_0^t
       \| \sigma (s, u_{v_n}(s))-
        \sigma (s,  \tilde{u} (s))\|_{\call_2(l^2, H)}
        \|
         v_n(s)\|_{l^2}   ds
         $$
         $$
          \le \left (
          \int_0^t
       \| \sigma (s, u_{v_n}(s))-
        \sigma (s,  \tilde{u} (s))\|^2_{\call_2(l^2, H)}
        ds
        \right )^{\frac 12}
        \|
         v_n \|_{L^2(0,T; l^2) }   ds
         $$
         $$
          \le  c_6 \left (
          \int_0^t 
          (
          1+ \|u_{v_n} (s)\|^{\frac p2}
          _{L^p(\R^n)}
          +
          \|\tilde{u}  (s)\|^{\frac p2}
          _{L^p(\R^n)}
          )  \|u_{v_n} (s) -\tilde{u} (s)\|
        ds
        \right )^{\frac 12} 
         $$
         $$
          \le  c_6 \left (
          \int_0^T
         3 (
          1+ \|u_{v_n} (s)\|^{p}
          _{L^p(\R^n)}
          +
          \|\tilde{u}  (s)\|^{p}
          _{L^p(\R^n)}
          )  
        ds
        \right )^{\frac 14} \|u_{v_n}  -\tilde{u}  \|
        ^{\frac 12} _{L^2(0,T; H)}
         $$
         $$
          \le  c_7    \|u_{v_n}  -\tilde{u}  \|
        ^{\frac 12} _{L^2(0,T; H)},
         $$
       where $c_7=c_7(T)>0$, which along with 
       \eqref{wc 41} yields \eqref{wc 43}.

         One can  also 
      verify  that
       for every
       $t\in [0,T]$,
              \be\label{wc 44}
      \lim_{n\to \infty}
       \int_0^t
         \sigma (s, \tilde{u}  (s))  
         v_n(s)   ds  =  \int_0^t
         \sigma (s, \tilde{u}  (s))  
         v (s)   ds
         \ \text{ weakly  in } \ H.
         \ee
         Indeed,  for every $t\in [0,T]$
         and $v_0\in L^2(0,T; l^2)$,
         by \eqref{sig6} we have
          $$
            \int_0^t \|
         \sigma (s, \tilde{u}  (s))  
         v_0 (s)     \| ds 
           \le
         \left ( \int_0^T \| \sigma (s, \tilde{u}  (s))\|^2
         _{\call_2(l^2,H)}  ds
         \right )^{\frac 12}
           \| v_0   \|_{L^2(0,T; l^2)}^2   
           $$
      $$ 
          \le
         \left ( \int_0^T
         \left (c_8  +  \| \tilde{u} (s) \|^p_{L^p(\R^n)}
         +2 \| \sigma_1  \|^2_{L^\infty(0,T;  L^2(\R^n,
          l^2)) } \right )
          ds
         \right )^{\frac 12}  \| v_0   \|_{L^2(0,T; l^2)}^2   
           $$
         \be\label{wc 45}
           \le c_9 \| v_0   \|_{L^2(0,T; l^2)}^2 ,
        \ee
           where $c_9=c_9   (T)>0$.

           Consider the operator   
         $\calg: L^2(0,T; l^2)
         \to H$  given by
         $$
         \calg (v_0)
         =\int_0^t  \sigma (s, \tilde{u}  (s))  
         v_0 (s)   ds,
         \quad \forall \  v_0
         \in L^2(0,T; l^2).
         $$
         Then \eqref{wc 45}
         indicates that
           $\calg 
         $ is a linear bounded operator 
         and  thus  it is weakly continuous,
         which implies  \eqref{wc 44}
         due to  
         $v_n \to v$ weakly in 
         $L^2(0,T; l^2)$.

         By \eqref{wc 43} and \eqref{wc 44}  we get
        for all $t\in [0,T]$ and
        $\xi \in V\bigcap L^p(\R^n)$,
    \be\label{wc 46} 
     \lim_{n\to \infty}
      \int_0^t 
  (\sigma (s, u_{v_n}(s)) v_n(s), \xi)  ds
  =   \int_0^t  
  (\sigma (s, \tilde{u}  (s)) 
   v(s), \ \xi)   ds.
   \ee

         On the other hand,
         by the
         standard argument
         (see, e.g., \cite{wangJDE2023}), we can obtain from
         \eqref{wc 10} that
           $\chi = F(\cdot, \cdot, \tilde{u} )$ and thus
         \be\label{wc 47}
         F(\cdot , \cdot, u_{v_n})
         \to F(\cdot, \cdot ,  \tilde{u} )
             \ \text{ weakly in  } \
         L^q(0,T; L^q(\R^n)).
       \ee

    Letting $n\to \infty$ in \eqref{wc 42},
   it follows from  
    \eqref{wc 7}-\eqref{wc 8} and
         \eqref{wc 46}-\eqref{wc 47} 
         that 
         for all $0 \le t\le T$
       and $\xi \in  V\bigcap L^p(\R^n) $,
  $$
  (\tilde{u}  (t), \xi)
  +\int_0^t 
  ( (-\Delta)^{\frac {\alpha}2}   \tilde{u} (s),
  (-\Delta)^{\frac {\alpha}2}  \xi
  ) ds
  +\int_0^t \int_{\R^n} F(s,x,  \tilde{u}  (s))
  \xi (x) dx  ds
  $$
  $$
  =(u_0,\xi)  +\int_0^t (g(s),\xi)  ds
  +\int_0^t 
  (\sigma (s,  \tilde{u}  (s)) v (s), \xi)  ds,
$$  
and hence
 $\tilde{u}$ is the solution of \eqref{contr1}-\eqref{contr2},
 that is, $\tilde{u} =u_v$.
 Then \eqref{wc 40} follows from \eqref{wc 41}
 immediately.

  {\bf Step 4: convergence
  of $\{u_{v_n}\}_{n=1}^\infty$
  in  $C([0,T], H) \bigcap L^2(0,T; V)
  \bigcap L^p(0,T; L^p(\R^n))$}.   We finally prove:
          \be\label{wc 50}
         u_{v_n}      \to  u_v
         \ \text{ strongly in } \
         C([0,T], H) \bigcap L^2(0,T; V)
         \bigcap L^p(0,T; L^p(\R^n)).
         \ee
    By \eqref{f3} and \eqref{contr1}-\eqref{contr2}
     we  get
       $$
     {\frac d{dt}}
     \|  u_{v_n}   (t) - u_v  (t)\|^2
     +2  \| (-\Delta)^{\frac {\alpha}2}
     ( u_{v_n}   (t) - u_v  (t) )\|^2
     $$
     $$
     + 2\lambda_2
     \int_{\R^n}
     \left (
     |u_{v_n}(t,x)|^{p-2} u_{v_n} (t,x)
     -
          |u_{v }(t,x)|^{p-2} u_{v } (t,x)
          \right ) (u_{v_n} (t,x)-u_{v } (t,x)) dx
     $$
     $$
    \le  2 \int_{\R^n} \psi_4 (t,x) |u_{v_n} (t,x)-u_{v } (t,x)|^2
    dx
      $$
     \be\label{wc 51}
     +2
     \left (
     \sigma (t,  u_{v_n} (t))v_n(t)-
     \sigma (t,  u_v(t))v (t),
     \   u_{v_n}   (t) - u_v  (t)
     \right ).
     \ee
      By \eqref{elity1}
      and \eqref{wc 51} we have
      for all $t\in [0,T]$,
        $$ 
     \|  u_{v_n}   (t) - u_v  (t)\|^2
     +2  \int_0^t \| (-\Delta)^{\frac {\alpha}2}
     ( u_{v_n}   (s) - u_v  (s) )\|^2 ds
     $$
     $$
     + 2^{2-p} \lambda_2\int_0^t 
       \|u_{v_n} (s)-u_{v } (s )\|^p_{L^p(\R^n)}  ds
     $$
     $$
     \le  2\| \psi_4 \|_{L^\infty(0,T;
     L^\infty (\R^n))}
     \int_0^T\|  u_{v_n} (s )-u_{v } (s)\|^2
    ds
      $$
     \be\label{wc 52}
     +2\int_0^T
     \left |\left (
     \sigma (s,  u_{v_n} (s))v_n(s)-
     \sigma (s,  u_v(s))v (s),
     \   u_{v_n}   (s) - u_v  (s)
     \right )  \right | ds.
     \ee

     By \eqref{sig8} and \eqref{wc 1}, for the last term
     in \eqref{wc 52}, we have for all $t\in [0,T]$,
      $$
       2\int_0^T
     \left |\left (
     \sigma (s,  u_{v_n} (s))v_n(s)-
     \sigma (s,  u_v(s))v (s),
     \   u_{v_n}   (s) - u_v  (s)
     \right )  \right | ds
      $$  
      $$
      \le
       2\int_0^T
     \|\sigma (s,  u_{v_n} (s))\|_{\call_2(l^2, H)}
     \|v_n(s)\|_{l^2}  
     \|u_{v_n}   (s) - u_v  (s)\| ds
     $$
      $$
     +  2\int_0^T
     \|\sigma (s,  u_{v } (s))\|_{\call_2(l^2, H)}
     \|v (s)\|_{l^2}  
     \|u_{v_n}   (s) - u_v  (s)\| ds
     $$
     $$
      \le
       2
        \|v_n\|_{L^2(0,T; l^2)}
       \left (\int_0^T
     \|\sigma (s,  u_{v_n} (s))\|^2_{\call_2(l^2, H)}  
     \|u_{v_n}   (s) - u_v  (s)\|^2 ds
     \right )^{\frac 12}
     $$
      $$
     +  
     2
        \|v\|_{L^2(0,T; l^2)}
       \left (\int_0^T
     \|\sigma (s,  u_{v} (s))\|^2_{\call_2(l^2, H)}  
     \|u_{v_n}   (s) - u_v  (s)\|^2 ds
     \right )^{\frac 12}
      $$
      $$
      \le
       2
        \|v_n\|_{L^2(0,T; l^2)}
       \left (\int_0^T 
       \left (c _{10} + c_{10} 
       \left (1+ \| u_{v_n} (s) 
       \|^{\frac p2}_{L^p(
       \R^n)}  \right )  \| u_{v_n} (s)\| \right  )
     \|u_{v_n}   (s) - u_v  (s)\|^2 ds
     \right )^{\frac 12}
     $$
      $$
     +  
     2
        \|v\|_{L^2(0,T; l^2)}
       \left (\int_0^T 
       \left (c _{10} + c_{10} 
       \left (1+ \| u_{v} (s) \|^{\frac p2}_{L^p(
       \R^n)}  \right ) \| u_{v} (s)\| \right )
     \|u_{v_n}   (s) - u_v  (s)\|^2 ds
     \right )^{\frac 12}
      $$
      $$
      \le
       2
        \|v_n\|_{L^2(0,T; l^2)} 
       \left (\int_0^T 
       \left 
       (c _{10} + c_{10} 
       \left (1+ \| u_{v_n} (s) \|^{\frac p2}_{L^p(
       \R^n)}  \right )
       \| u_{v_n}  \| _{C([0,T],
        H)}\right )
     \|u_{v_n}   (s) - u_v  (s)\|^2 ds
     \right )^{\frac 12}
     $$
      $$
     +  
     2
        \|v\|_{L^2(0,T; l^2)}
       \left (\int_0^T 
       \left (c _{10} + c_{10} 
       \left (1+ \| u_{v} (s) \|^{\frac p2}_{L^p(
       \R^n)} \right ) \| u_{v}  \|_{C([0,T], H)}
       \right )
     \|u_{v_n}   (s) - u_v  (s)\|^2 ds
     \right )^{\frac 12}
      $$
      $$
      \le
      c_{11}
       \left (\int_0^T 
        \left (1+ \| u_{v_n} (s) \|^{\frac p2}_{L^p(
       \R^n)}  \right ) 
     \|u_{v_n}   (s) - u_v  (s)\|^2 ds
     \right )^{\frac 12}
     $$
      $$
     +  
     c_{11}
       \left (\int_0^T 
        \left (1+ \| u_{v} (s) \|^{\frac p2}_{L^p(
       \R^n)}  \right ) 
     \|u_{v_n}   (s) - u_v  (s)\|^2 ds
     \right )^{\frac 12}
      $$
      $$
      \le
      c_{11}
       \left (\int_0^T 
        \left (2+ 2\| u_{v_n} (s) \|^{p}_{L^p(
       \R^n)}  \right )  ds \right )^{\frac 14}
     \|u_{v_n}   - u_v  \|_{L^4(0,T; H)}  
       $$
       $$+
      c_{11}
       \left (\int_0^T 
        \left (2+ 2\| u_{v} (s) \|^{p}_{L^p(
       \R^n)}  \right )  ds \right )^{\frac 14}
     \|u_{v_n}   - u_v  \|_{L^4(0,T; H)}  
       $$
      $$
      \le
      c_{11} 
        \left (2T + 2 c_1^p 
        \right )^{\frac 14}
     \|u_{v_n}   - u_v  \|_{L^4(0,T; H)}  
     $$
    \be\label{wc 53}
        +
      c_{11}
       \left (\int_0^T 
        \left (2+ 2\| u_{v} (s) \|^{p}_{L^p(
       \R^n)}  \right )  ds \right )^{\frac 14}
     \|u_{v_n}   - u_v  \|_{L^4(0,T; H)} , 
     \ee
      where $c_{11} =c_{11}(T)>0$.
      
      It follows  from \eqref{wc 40} and
      \eqref{wc 53} that
     \be\label{wc 55}
      \lim_{n\to \infty}
       2\int_0^T
     \left |\left (
     \sigma (s,  u_{v_n} (s))v_n(s)-
     \sigma (s,  u_v(s))v (s),
     \   u_{v_n}   (s) - u_v  (s)
     \right )  \right | ds =0.
     \ee
      Then \eqref{wc 50} follows from
      \eqref{wc 40}  and
      \eqref{wc 55} immediately.
      This completes the proof. 
   \end{proof}

    In order to  
    investigate the uniform LDP
    of solutions of \eqref{intr1}-\eqref{intr3},
    for every $u_0\in H$,
    we need an operator
   $\calg^0_{u_0}: C([0,T], U) \to C([0,T], H)
  \bigcap L^2(0,T; V)
  \bigcap
  L^p(0,T; \R^n)$ which is defined by,
   for  $\xi \in  C([0,T], U)$, 
  \be\label{calg0}
\calg^0_{u_0} (\xi)
=
\left \{
\begin{array}{ll}
u_v(\cdot, u_0),  & \text{ if } \xi= \int_0^\cdot v(t) dt
\ \text{ with } v\in L^2(0,T; l^2);\\
0, &  \text{ otherwise} ,
\end{array}
\right.
\ee
 where $u_v(\cdot, u_0)$ is the
 solution
 of \eqref{contr1}-\eqref{contr2}
 defined on $[0,T]$
 with initial value $u_0$.

  For every  $u_0\in H$, 
  define a rate function
  $I_{u_0}:
    C([0,T], H)
  \bigcap L^2(0,T; V) \bigcap
  L^p(0,T; \R^n) \to [0, +\infty]$ by, 
  for each $\phi \in   C([0,T], H)
  \bigcap L^2(0,T; V) \bigcap
  L^p(0,T; \R^n)  $,
\be\label{rate}
 I_{u_0} (\phi)
 =\inf
 \left \{
 {\frac 12} \int_0^T \| v(s)\|_{l^2} ^2 ds:
 \ v\in L^2(0,T; l^2), \ u_v(\cdot, u_0)  =\phi
 \right \},
\ee
 where $u_v(\cdot, u_0)$ is the
 solution
 of \eqref{contr1}-\eqref{contr2}.
By default,   $\inf \emptyset =+\infty$.
 
 By Lemma \ref{wc}, one can verify that
 the rate function $I_{u_0}$ as given by
 \eqref{rate} is a good rate function, and
 the infimum of \eqref{rate} is actually a minimum
 when $I_{u_0} (\phi)<\infty$.
 Moreover, 
   for every $N<\infty$, the set
\be\label{comg0}
 \left \{
\calg^0_{u_0}
\left (
\int_0^\cdot v(t) dt
\right ) : \  v\in {\overline{B}}_N(L^2(0,T; l^2))
\right \}
\ \ \text{is a compact subset
of  } \ 
\ee
 $ C([0,T], H)
  \bigcap L^2(0,T; V)\bigcap L^p
  (0,T; L^p(\R^n))$.

 \begin{lem}\label{geps}
  If  \eqref{f1}-\eqref{f3}
  and \eqref{sig1}-\eqref{sig4}
  are fulfilled,  then for every
  $u_0\in H$,  $T>0$    and
 $v\in \cala_N$
with $N<\infty$,  the process
 $u^\eps_v
=
 \calg^\eps_{u_0}
\left (
W +\eps^{-\frac 12}\int_0^\cdot
v (t) dt
\right )$ satisfies  the equation
 \be\label{geps 1}
  du_v^{\eps}  
  + (-\Delta)^ \alpha  u_v^{\eps}   dt
  + F(t, \cdot, u_v^{\eps} ) dt
  =  
    g(t)   dt
    +  \sigma (t,   u_v^{\eps} )    v dt
  +\sqrt{\eps} \sigma (t,   u_v^{\eps} )    {dW} ,
  \ \ t\in (0,T),
  \ee  
 with  initial condition 
 $
 u^\eps_{v} (0)=u_0\in H. $

  In addition,  
       for all
       $u_0\in\overline{B}_R(H)$,
           and  
    $v \in \cala_N$,
     the process $u_{v}^\eps$   enjoys the uniform estimates:  
  \be\label{geps 2}
 \E\left (
 \| u^\eps_v \|^2_{C([0,T], H)}
 +  \| u^\eps_v \|^2_{L^2(0,T;V )}
 + \| u^\eps_v \|^p_{L^p(0,T; L^p(\R^n))}
 \right )
 \le L_3,
 \ee
 where $L_3>0$ depends on $R,T,N$, but not on $u_0$,
 $v$ or $\eps\in (0,1)$.
 \end{lem}

 \begin{proof}
  It follows from
  Girsanov\rq{}s theorem
  that  
$ 
W +\eps^{-\frac 12}\int_0^\cdot
v (t) dt$
is also a 
cylindrical Wiener
process in $l^2$,  and hence 
   $u_{v}^\eps
 =\calg^\eps_{u_0}
 (W +\eps^{-\frac 12}\int_0^\cdot
v (t) dt )
  $
  is the unique solution of
\eqref{geps 1} with initial
value
$
 u^\eps_{v} (0)=u_0$.

By  \eqref{f1}, \eqref{geps} and Ito\rq{}s  formula we  get
for $t\in [0,T]$,
$$
\| {u_v}^\eps(t) \|^2
+2\int_0^t
\| (-\Delta)^{\frac {\alpha}2}  u_{v}^\eps
(s) \|^2 ds
+
  2  \lambda_1 \int_0^t 
  \|  u_{v}^\eps (s)\|^p_{L^p(\R^n)}ds
  $$
 $$
\le
 \| u_0 \|^2  + 2\|\psi_1 \|_{L^1(0,T; L^1(\R^n))}
 +2\int_0^t ( u_{v}^\eps(s), g(s)) ds
 +2\int_0^t
 (u_{v}^\eps (s), \sigma (s, u_{v}^\eps (s)) v(s)) ds
 $$
\be\label{geps p1}
 +\eps\int_0^t
 \| \sigma (s, u_{v}^\eps (s))\|^2_{\call_2(l^2, H)} ds
 + 2\sqrt{\eps}
 \int_0^t (u_{v}^\eps (s),
 \sigma (s, u_{v}^\eps (s)) dW).
 \ee
 By \eqref{sig6}, we have,
 for all $t\in [0,T]$ and
 $\eps\in (0,1)$, 
 $$
 2 \int_0^t | ( u_{v}^\eps (s),   \sigma (s,  u_{v}^\eps (s)) v(s) )|
 ds
 + \eps\int_0^t
 \| \sigma (s, u_{v}^\eps (s))\|^2_{\call_2(l^2, H)} ds
 $$
 $$
 \le
 \int_0^t
 \|v(s)\|^2_{l^2}
 \| u_{v}^\eps (s)\|^2 ds
 +  
 2 \int_0^t
 \| \sigma (s, u_{v}^\eps (s))\|^2_{\call_2(l^2, H)} ds
 $$
  \be\label{geps p2}
 \le
 \int_0^t
 \|v(s)\|^2_{l^2}
 \| u_{v}^\eps (s)\|^2 ds
 +  
  {\frac 12} \lambda_1 
 \int_0^t \|  u_{v}^\eps
 (s) 
 \|^p_{L^p(\R^n)} ds
 +4 \|\sigma_1 \|^2
 _{L^2(0,T;L^2(\R^n,l^2))}
 + c_1 T,
\ee
where $c_1>0$ depends on
$\lambda_1$, $p,q, \kappa$,
$\sum_{k=1}^\infty\beta_k$
and
$\sum_{k=1}^\infty\gamma_k$.
  By \eqref{geps p1}-\eqref{geps p2} we get
   for $\eps\in (0,1)$ and $t\in [0, T]$,
    $$
\| u_{v}^\eps(t) \|^2
+2\int_0^t
\| (-\Delta)^{\frac {\alpha}2}  u_{v}^\eps
(s) \|^2 ds
+
{\frac 32}
\lambda_1
\int_0^t \| u_{v}^\eps (s) \|^p_{L^p(\R^n)} ds
 $$
\be\label{geps p3}
 \le
 \| u_0 \|^2
 +  \int_0^t
 (1+ \| v(s)\|^2_{l^2}) \| u_{v}^\eps (s) \|^2ds 
  + 2 \sqrt{\eps}
  \int_0^t (u_{v}^\eps (s),
 \sigma (s, u_{v}^\eps (s)) dW)
    +c_2
 \ee
 where $c_2=2\|\psi_1\|_{L^1
 (0,T; L^1(\R^n))}
+ \|g\|^2_{L^2(0,T; H)}
+ 4\|\sigma_1\|^2
_{L^2
(0,T; L^2(\R^n,l^2))}
+ c_1T$.
Then for all 
   $u_0\in H$ with $\| u_0\|\le R$,
   we obtain
   from
   \eqref{geps p3}
   that  
   for $\eps\in (0,1)$ and $t\in [0, T]$,
   $$
   \| u_{v}^\eps(t) \|^2
   +2\int_0^t
   \| (-\Delta)^{\frac {\alpha}2}  u_{v}^\eps
   (s) \|^2 ds
   +
   {\frac 32}
   \lambda_1
   \int_0^t \| u_{v}^\eps (s) \|^p_{L^p(\R^n)} ds
   $$
   \be\label{geps p4}
   \le 
      \int_0^t
   (1+ \| v(s)\|^2_{l^2}) \| u_{v}^\eps (s) \|^2ds 
     +c_2 +R^2
     + M^\eps (t),
   \ee
 where 
 $ 
  M^\eps (t) = 2 \sqrt{\eps} 
 \int_0^t (u_{v}^\eps (s),
 \sigma (s, u_{v}^\eps (s)) dW)$.
  By  \eqref{geps p4}
  and Gronwall\rq{}s inequality, we see that   
  for all 
  $\eps\in (0,1)$,
  $t\in [0,T]$ and $v\in \cala_N$, 
  $$
 \| u_{v}^\eps(t) \|^2
+2\int_0^t
\| (-\Delta)^{\frac {\alpha}2}  u_{v}^\eps
(s) \|^2 ds
+
{\frac 32}
\lambda_1
\int_0^t \| u_{v}^\eps (s) \|^p_{L^p(\R^n)} ds
 $$
$$
 \le (c_2 +R^2)
 e^{
 \int_0^t
 (1+ \| v(r)\|^2_{l^2}) dr
 }  
 + 
 \int_0^t
 e^{ 
 \int_s^t
 (1+ \| v(r)\|^2_{l^2}) dr
 }
 dM^\eps (s)   
$$
$$
 \le ( c_2  +R^2)
 e^{ T+N^2} 
 + 
 M^\eps (t)
 +  \int_0^t
  (1+ \| v(s)\|^2_{l^2})
 e^{  
 \int_s^t
 (1+ \| v(r)\|^2_{l^2}) dr
 }
 M^\eps (s) ds  
$$
$$
 \le 
 ( c_2  +R^2)
 e^{ T+N^2}  
 + 
 M^\eps (t)
 + 
 (\sup_{0\le s \le t}
 |M^\eps (s)|) \  e^{  
 \int_0^T
 (1+ \| v(r)\|^2_{l^2}) dr
 }
  \int_0^T
  (1+ \| v(s)\|^2_{l^2})
    ds  
$$
  \be\label{geps p5} 
\le c_3 + c_3
\left (\sup_{0\le s \le t}
 |M^\eps (s)| \right ),
 \ee
 where $c_3=c_3(R, T, N)>0$.
 It follows from
 \eqref{geps p5} that
  for all 
 $\eps\in (0,1)$,
 $t\in [0,T]$ and $v\in \cala_N$, 
 $$
 \sup_{0\le r\le t}
 \| u_{v}^\eps(r) \|^2
 +2\int_0^t
 \| (-\Delta)^{\frac {\alpha}2}  u_{v}^\eps
 (s) \|^2 ds
 +
 {\frac 32}
 \lambda_1
 \int_0^t \| u_{v}^\eps (s) \|^p_{L^p(\R^n)} ds
 $$
 $$ 
 \le 2c_3 +  2c_3
 \left (\sup_{0\le s \le t}
 |M^\eps (s)| \right ),
$$
which shows that 
  for all 
 $\eps\in (0,1)$,
 $t\in [0,T]$ and $v\in \cala_N$, 
 $$
 \E\left (  \sup_{0\le r\le t}
 \| u_{v}^\eps(r) \|^2
 \right )
 +2\int_0^t
 \E\left (
  \| (-\Delta)^{\frac {\alpha}2}  u_{v}^\eps
 (s) \|^2 \right )  ds
 +
 {\frac 32}
 \lambda_1
 \int_0^t 
 \E \left (
  \| u_{v}^\eps (s) \|^p_{L^p(\R^n)} \right ) ds
 $$
\be\label{geps p6}
 \le 2c_3 +  2c_3
 \E \left (\sup_{0\le s \le t}
 |M^\eps (s)| \right ).
 \ee

 By
 \eqref{sig6}  and 
  the Burkholder inequality, we 
  get for all
  $\eps\in (0,1)$  and
  $t\in [0,T]$,
 $$
 2c_3 \E \left (\sup_{0\le s \le t}
 |M^\eps (s)|
 \right )
  \le
  12 c_3 \E \left ( 
  \left (\int_0^t
  \| u_{v}^\eps (s)\|^2
 \| \sigma (s, u_{v}^\eps (s)  )\|^2_{\call_2(l^2, H)}
 ds 
 \right )^{\frac 12}
 \right )
 $$
 $$
 \le
 {\frac 12}
 \E \left (
  \sup_{0\le s \le t}
  \| u_{v}^\eps (s)\| ^2
 \right )
 + 72 c_3^2
   \E \left (  
  \int_0^t
   \| \sigma (s, u_{v}^\eps (s)  )\|^2_{\call_2(l^2, H)}
 ds  
 \right )
 $$
 $$
 \le
 {\frac 12}
 \E \left (
 \sup_{0\le s \le t}
 \| u_{v}^\eps (s)\| ^2
 \right )
 +{\frac 12} \lambda_1
 \int_0^t
 \E\left (
 \|u_{v}^\eps (s) \|^p_{L^p(\R^n)}
 \right ) ds
 + 144 c_3^2 \|
 \sigma_1\|^2
 _{L^2(0,T;
 	L^2(\R^n, l^2 )}
 +c_4
 $$
 for some $c_4>0$
 depending  on $\lambda_1$,
 $T$, 
 $p,  q, \kappa$,
$\sum_{k=1}^\infty \beta_k$
and
$\sum_{k=1}^\infty \gamma_k$, 
 which along with
 \eqref{geps p6} shows that
   for all 
  $\eps\in (0,1)$,
  $t\in [0,T]$, $u_0\in\overline{B} _R(H)$  and $v\in \cala_N$, 
  $$
 {\frac 12}
  \E\left (  \sup_{0\le r\le t}
  \| u_{v}^\eps(r) \|^2
  \right )
  +2\int_0^t
  \E\left (
  \| (-\Delta)^{\frac {\alpha}2}  u_{v}^\eps
  (s) \|^2 \right )  ds
  + 
  \lambda_1
  \int_0^t 
  \E \left (
  \| u_{v}^\eps (s) \|^p_{L^p(\R^n)} \right ) ds
  $$
  \be\label{geps p7}
  \le 2c_3   
  + 144 c_3^2 \|
  \sigma_1\|^2
  _{L^2(0,T;
  	L^2(\R^n, l^2 )}
  +c_4.
  \ee
 Then
 \eqref{geps 2} follows from
 \eqref{geps p7} immediately.
  \end{proof}

    The following lemma
    is concerned with
  the convergence of solutions
 of \eqref{intr1}-\eqref{intr3} as
 $\eps \to 0$.

 \begin{lem}\label{cvs}
 	 If  \eqref{f1}-\eqref{f3}
 	and \eqref{sig1}-\eqref{sig4}
 	are fulfilled,  then for every
 	   $R>0$, $N>0$    and
 	$\eta>0$,  
 $$
 \sup_{\|u_0\|
 	\le R}
 \sup_{v\in\cala_N}
 P
 \left ( 
  \|	\calg^\eps_{u_0}
 	  (
 	W +\eps^{-\frac 12}\int_0^\cdot
 	v (t) dt
   )
 	-
 	\calg^0_{u_0}
   ( \int_0^\cdot
 	v (t) dt
 	  )\|
 	  _{
 	 C([0,T], H) \bigcap L^2(0,T; V)
 	 \bigcap L^p(0,T;
 	 L^p(\R^n) )
} >\eta 
\right )
 	  $$  
 	  $$\longrightarrow  0,
 	  \quad\text{ as } \eps \to 0.
 	  $$
  \end{lem}

\begin{proof}
	Given $u_0\in H$
	with $\|u_0\| \le R$
	and $v\in \cala_N$,
 let 
 $u_{v}^\eps 
 =
 \calg^\eps_{u_0} 
\left (
W +\eps^{-\frac 12}\int_0^\cdot
v (t) dt
\right )$
and
 $u_v=
 \calg^0_{u_0}
 \left (
 \int_0^\cdot
 v(t) dt
 \right )   $. 
 Then 
  $u_{v}^\eps $ and
 $u_v$
 are the solutions of
 \eqref{geps 1} and  
  \eqref{contr1} 
  with  $u_{v}^\eps 
  (0) =u_v (0) =u_0$,
  respectively,
  and hence 
    $u_{v}^\eps -
     u_{v}$ satisfies
   $$
 d  ( u_{v}^\eps -  u_{v}  )
 +     (-\Delta)^ \alpha ( u^\eps_{v}  -  u_{v  }
 ) dt 
 +
   ( F(t,\cdot, u^\eps_{v}  ) -  
   F(t, \cdot, u_{v}  ) ) dt   
$$
  \be\label{cvs p1}
  =
\left ( 
 \sigma(t, u^\eps_{v}  )  v
 -
  \sigma(t, u_{v}  )  v
\right )  dt
 + \sqrt{\eps} \sigma(t, u^\eps_{v}  )  
   dW  ,
\ee
with  
$u_{v}^\eps(0) -  u_{v} (0)=0$.
 By  \eqref{f3},
 \eqref{cvs p1} and It\^{o}\rq{}s formula we obtain
 for all $t\in [0,T]$,
$$
 \|  u_{v}^\eps(t) -  u_{v}(t) \|^2
+2 
\int_0^t
\|  (-\Delta)^{\frac {\alpha}2}
( u^\eps_{v}(s)  -  u_{v} (s)
 ) \|^2  ds
 $$
 $$
 +2\lambda_1\int_0^t
 \int_{\R^n}
 \left (
 |u^\eps_{v}(s)| ^{p-2} u^\eps_{v}(s)
 -
  |u_{v}(s)| ^{p-2} u_{v}(s)
  \right )( u^\eps_{v}(s)  -  u_{v} (s)
 ) dx ds
  $$
 $$
\le 2 \|\psi_4\|_{L^\infty(0,T;
L^\infty (\R^n))}
\int_0^t
\| u^\eps_{v}(s)  -  u_{v} (s)
\|^2 ds
  $$
  $$ 
 +
  2 
  \int_0^t
  \left ( 
 \sigma(s, u^\eps_{v} (s) )  v(s)
 -
  \sigma(s, u_{v} (s)  )  v (s),
  \  u^\eps_{v}(s)  -  u_{v}(s)
\right )  ds
$$
   \be\label{cvs p2}
+ \eps
\int_0^t
 \|   \sigma(s, u^\eps_{v} (s)  )\|
_{\call_2(l^2, H)}^2  ds
+ 
2\sqrt{\eps} \int_0^t
\left (
u^\eps_{v}(s)  -  u_{v}(s),\
 \sigma(s, u^\eps_{v} (s)  ) 
   dW  \right ) .
\ee

  By \eqref{elity2}  
  and \eqref{sig7}
    we have
$$  
  2 
  \int_0^t
  \left |
  \left ( 
 \sigma(s, u^\eps_{v} (s) )  v(s)
 -
  \sigma(s, u_{v} (s)  )  v (s),
  \  u^\eps_{v}(s)  -  u_{v}(s)
\right ) 
\right |  ds
$$
$$  
 \le   
  \int_0^t
  \|
 \sigma(s, u^\eps_{v} (s) )  
 -
  \sigma(s, u_{v} (s)  )
  \|_{\call_2(l^2, H)}^2
  +
  \int_0^t
  \| v (s) \|_{l^2}^2 \|
  u^\eps_{v}(s)  -  u_{v}(s)
  \| ^2  ds
$$
$$  
 \le   {\frac 12}\lambda_1
  \int_0^t\int_{\R^n}
  \left (
  |u^\eps_{v} (s) |^{p-2}
  +
    |u_{v} (s) |^{p-2}
  \right )|
  u^\eps_{v}(s)  -  u_{v}(s)
  |^2  dx ds
  $$
   $$
  +
  \int_0^t
  \left (c_1+
  \| v (s) \|_{l^2}^2
  \right ) \|
  u^\eps_{v}(s)  -  u_{v}(s)
  \| ^2  ds
$$
$$  
 \le    \lambda_1
  \int_0^t\int_{\R^n}
  \left (
  |u^\eps_{v} (s) |^{p-2}u^\eps_{v} (s)
   -
     |u_{v} (s) |^{p-2}u_{v} (s)
   \right ) \left (
  u^\eps_{v}(s)  -  u_{v}(s)
 \right )  dx ds
  $$
  \be\label{cvs p3}
  +
  \int_0^t
  \left (c_1+
  \| v (s) \|_{l^2}^2
  \right ) \|
  u^\eps_{v}(s)  -  u_{v}(s)
  \| ^2  ds,
\ee
where $c_1>0$ depends only on
$\lambda_1$,
$p, q, \kappa$
and $\sum_{k=1}^\infty 
\alpha_k$.
By \eqref{cvs p2}-\eqref{cvs p3} we obtain
for all $t\in [0,T]$,
$$
 \|  u_{v}^\eps(t) -  u_{v}(t) \|^2
+2 
\int_0^t
\|  (-\Delta)^{\frac {\alpha}2}
( u^\eps_{v}(s)  -  u_{v} (s)
 ) \|^2  ds
 $$
 $$
 + \lambda_1\int_0^t
 \int_{\R^n}
 \left (
 |u^\eps_{v}(s)| ^{p-2} u^\eps_{v}(s)
 -
  |u_{v}(s)| ^{p-2} u_{v}(s)
  \right )( u^\eps_{v}(s)  -  u_{v} (s)
 ) dx ds
  $$
 $$
\le  
\int_0^t
\left (
c_2 + \| v(s)\|^2_{l^2}
\right )
\| u^\eps_{v}(s)  -  u_{v} (s)
\|^2 ds
  $$ 
   \be\label{cvs p4}
+ \eps
\int_0^t
 \|   \sigma(s, u^\eps_{v} (s)  )\|
_{\call_2(l^2, H)}^2  ds
+ 
2\sqrt{\eps} \int_0^t
\left (
u^\eps_{v}(s)  -  u_{v}(s),\
 \sigma(s, u^\eps_{v} (s)  ) 
   dW  \right ) ,
\ee
  where 
  $c_2= c_1 + 2 \|\psi_4\|_{L^\infty(0,T;
L^\infty (\R^n))}$.
By \eqref{elity1} and \eqref{cvs p4}
we have 
for all $t\in [0,T]$,
$$
 \|  u_{v}^\eps(t) -  u_{v}(t) \|^2
+2 
\int_0^t
\|  (-\Delta)^{\frac {\alpha}2}
( u^\eps_{v}(s)  -  u_{v} (s)
 ) \|^2  ds
 $$
 $$
 + 2^{1-p} \lambda_1\int_0^t
  \|  u^\eps_{v}(s)  -  u_{v} (s)
  \|^p_{L^p(\R^n)} ds
  $$
 $$
\le  
\int_0^t
\left (
c_2 + \| v(s)\|^2_{l^2}
\right )
\| u^\eps_{v}(s)  -  u_{v} (s)
\|^2 ds
  $$ 
   \be\label{cvs p5}
+ \eps
\int_0^t
 \|   \sigma(s, u^\eps_{v} (s)  )\|
_{\call_2(l^2, H)}^2  ds
+ 
2\sqrt{\eps} \int_0^t
\left (
u^\eps_{v}(s)  -  u_{v}(s),\
 \sigma(s, u^\eps_{v} (s)  ) 
   dW  \right ) .
\ee

  Given $K>0$,    denote by 
$$
 \tau^\eps_K
 =\inf \left \{
 t\ge 0: \| u^\eps_{v} (t) \| \ge K
 \right \}\wedge T.
$$
By \eqref{cvs p5} we have 
for all $t\in [0,T]$,
 $$ \sup_{0\le r \le t}
 \left (
 \|  u_{v}^\eps(r \wedge \tau^\eps_K ) 
 -  u_{v}( r \wedge \tau^\eps_K) \|^2 \right )
+2 
\int_0^{t \wedge \tau^\eps_K}
\|  (-\Delta)^{\frac {\alpha}2}
( u^\eps_{v}(s)  -  u_{v} (s)
 ) \|^2  ds
 $$
 $$
 + 2^{1-p} \lambda_1\int_0 ^{t \wedge \tau^\eps_K}
  \|  u^\eps_{v}(s)  -  u_{v} (s)
  \|^p_{L^p(\R^n)} ds
  $$
%  $$
%\le
%   2\int_0^{t \wedge \tau^\eps_K} (c_2
%   + \| v (s)\|^2_{l^2})
% \|  u^\eps_{v}(s)  -  u_{v}(s) \|^2
% ds
% $$
%$$
%+ 2\eps
%\int_0^{t \wedge \tau^\eps_K}
% \|   \sigma(s, u^\eps_{v} (s)  )\|
%_{\call_2(l^2, H)}^2  ds
%+ 
%4\sqrt{\eps}
%\sup_{0\le r \le t}
% \left|
%\int_0^{r \wedge \tau^\eps_K}
%\left (
%u^\eps_{v}(s)  -  u_{v}(s),\
% \sigma(s, u^\eps_{v} (s)  ) 
%   dW  \right ) 
%   \right |
%   $$
  $$
\le
 2 \int_0^{t  } (c_2 + \| v(s)\|^2_{l^2})
 \sup_{0\le r \le s}
 \|  u^\eps_{v}(r\wedge \tau^\eps_K) 
  -  u_{v}(r\wedge \tau^\eps_K) \|^2
 ds
 $$
   \be\label{cvs p6}
+ 2 \eps
\int_0^{T\wedge \tau^\eps_K}
 \|   \sigma(s, u^\eps_{v} (s)  )\|
_{\call_2(l^2, H)}^2  ds
+ 
4\sqrt{\eps}
\sup_{0\le r \le T}
 \left|
\int_0^{r \wedge \tau^\eps_K}
\left (
u^\eps_{v}(s)  -  u_{v}(s),\
 \sigma(s, u^\eps_{v} (s)  ) 
   dW  \right ) 
   \right |.
 \ee
By Gronwall\rq{}s inequality, we obtain
from \eqref{cvs p6} that
  for all $t\in [0,T]$,  
 $$
  \sup_{0\le r \le t}
 \left (
 \|  u_{v}^\eps(r \wedge \tau^\eps_K ) 
 -  u_{v}( r \wedge \tau^\eps_K) \|^2
 \right )
+2 
\int_0^{t \wedge \tau^\eps_K}
\|  (-\Delta)^{\frac {\alpha}2}
( u^\eps_{v}(s)  -  u_{v} (s)
 ) \|^2  ds 
 $$
 $$
 + 2^{1-p} \lambda_1\int_0 ^{t \wedge \tau^\eps_K}
  \|  u^\eps_{v}(s)  -  u_{v} (s)
  \|^p_{L^p(\R^n)} ds
  $$
$$
\le  2 \eps  e^{2\int_0^t (c_2 + \| v(s)\|^2_{l^2}) ds}
\int_0^{T\wedge \tau^\eps_K}
 \|   \sigma(s, u^\eps_{v} (s)  )\|
_{\call_2(l^2, H)}^2  ds
$$
  \be\label{cvs p7}
+ 
4 \sqrt{\eps}e^{2\int_0^t (c_2 + \| v(s)\|^2_{l^2}) ds}
\sup_{0\le r \le T}
 \left|
\int_0^{r \wedge \tau^\eps_K}
\left (
u^\eps_{v}(s)  -  u_{v}(s),\
 \sigma(s, u^\eps_{v} (s)  ) 
   dW  \right ) 
   \right | .
 \ee
 It follows from 
 \eqref{cvs p7} that
 for all $\|u_0\|\le R$ and $v\in \cala_N$,
 $$
 \E \left (
  \sup_{0\le r \le T}
 \left (
 \|  u_{v}^\eps(r \wedge \tau^\eps_K ) 
 -  u_{v}( r \wedge \tau^\eps_K) \|^2
 \right )
 \right )
+2  \E \left ( 
\int_0^{T\wedge \tau^\eps_K}
\|  (-\Delta)^{\frac {\alpha}2}
( u^\eps_{v}(s)  -  u_{v} (s)
 ) \|^2  ds 
 \right )
 $$
 $$
 + 2^{1-p} \lambda_1
 \E\left (
 \int_0 ^{T\wedge \tau^\eps_K}
  \|  u^\eps_{v}(s)  -  u_{v} (s)
  \|^p_{L^p(\R^n)} ds
  \right )
  $$
$$
\le  2 \eps 
\E \left (
 e^{2\int_0^T (c_2 + \| v(s)\|^2_{l^2}) ds}
\int_0^{T\wedge \tau^\eps_K}
 \|   \sigma(s, u^\eps_{v} (s)  )\|
_{\call_2(l^2, H)}^2  ds
\right )
$$
$$
+ 
4 \sqrt{\eps}
\E \left (
e^{2\int_0^T (c_2 + \| v(s)\|^2_{l^2}) ds}
\sup_{0\le r \le T}
 \left|
\int_0^{r \wedge \tau^\eps_K}
\left (
u^\eps_{v}(s)  -  u_{v}(s),\
 \sigma(s, u^\eps_{v} (s)  ) 
   dW  \right ) 
   \right |
   \right ) 
 $$ 
$$
\le  2 \eps 
 e^{2Tc_2 + 2N^2  }
\E \left (
\int_0^{T\wedge \tau^\eps_K}
 \|   \sigma(s, u^\eps_{v} (s)  )\|
_{\call_2(l^2, H)}^2  ds
\right )
$$
  \be\label{cvs p8}
+ 
4 \sqrt{\eps} e^{2Tc_2 + 2N^2  }
\E \left ( 
\sup_{0\le r \le T}
 \left|
\int_0^{r \wedge \tau^\eps_K}
\left (
u^\eps_{v}(s)  -  u_{v}(s),\
 \sigma(s, u^\eps_{v} (s)  ) 
   dW  \right ) 
   \right |
   \right ) .
 \ee

We now deal with the right-hand side of
\eqref{cvs p8}.
By Lemmas \ref{cosol} and \ref{geps} we see that
there exists $c_3 =c_3(R,N,T)>0$
independent of $u_0$, $v$ and $\eps$ such that
for all $\|u_0\|\le R$,
$v\in \cala_N$ and $\eps\in (0,1)$,
\be\label{cvs p10}
\|u_v\|^2_{C([0,T], H)}
+\|u_v\|^2_{L^2(0,T; V)}
+\|u_v\|^p
_{L^p(0,T; L^p(\R^n))}
\le c_3,
\quad
\text{ 
P-almost surely}, 
\ee and
\be\label{cvs p11}
\E \left ( \|u^\eps_v\|^2_{C([0,T], H)}
\right )
+\E \left (
\|u^\eps_v\|^2_{L^2(0,T; V)}
\right )
+\E \left ( 
\|u^\eps_v\|^p
_{L^p(0,T; L^p(\R^n))}
\right )
\le c_3.
\ee

By  \eqref{sig6} and \eqref{cvs p11} we have
for all $\|u_0\|\le R$,
$v\in \cala_N$ and $\eps\in (0,1)$,
 $$
 \E \left (
\int_0^{T\wedge \tau^\eps_K}
 \|   \sigma(s, u^\eps_{v} (s)  )\|
_{\call_2(l^2, H)}^2  ds
\right )
\le
 \E \left (
\int_0^ { T }
 \|   \sigma(s, u^\eps_{v} (s)  )\|
_{\call_2(l^2, H)}^2  ds
\right )
$$
$$
\le
\E \left (
\int_0^T
  \| u^\eps_{v} (s) \|^p_{L^p(\R^n)} 
   ds 
   \right ) 
+ 2\| \sigma_1\|^2_{L^2(0,T; ,L^2(\R^n,l^2))}
 + c_4T
 $$
 \be\label{cvs p12}
 \le   c_3
 +  2\| \sigma_1\|^2_{L^2(0,T; ,L^2(\R^n,l^2))}
 + c_4T,
\ee
where $c_4>0$ depends only on
$p, q, \kappa$,
$\sum_{k=1}^\infty \beta_k$
and
$\sum_{k=1}^\infty \gamma_k$.

  On the other hand, 
  by \eqref{cvs p10},
 \eqref{cvs p12} 
and
Doob\rq{}s maximal inequality we  obtain
that for all $\|u_0\|\le R$,
$v\in \cala_N$ and $\eps\in (0,1)$,
$$\E \left ( 
\sup_{0\le r \le T}
 \left|
\int_0^{r \wedge \tau^\eps_K}
\left (
u^\eps_{v}(s)  -  u_{v}(s),\
 \sigma(s, u^\eps_{v} (s)  ) 
   dW  \right ) 
   \right |^2
   \right )
   $$
  $$
\le 4
\E\left ( 
  \int_0^ { T\wedge \tau^\eps_K}
 \|
u^\eps_{v}(s)  -  u_{v}(s)\|^2
 \| \sigma(s, u^\eps_{v} (s)   )\|^2_
 {\call_2(l^2, H )}
 ds 
   \right )
$$
$$
\le  8
(K^2+c_3^2)
\E\left ( 
  \int_0^ { T\wedge \tau^\eps_R}
 \| \sigma(s, u^\eps_{v^\eps} (s)   )\|^2_{\call_2(
 l^2,H )}
 ds 
   \right )
$$
  \be\label{cvs p13}
\le 
 8
(K^2+c_3^2)
\left (
 c_3
 +  2\| \sigma_1\|^2_{L^2(0,T; ,L^2(\R^n,l^2))}
 + c_4T 
\right ).
\ee
  
By \eqref{cvs p8}
and \eqref{cvs p12}-\eqref{cvs p13}we
obtain that 
 for all  $\eps \in (0,1) $, 
 $\|u_0\|\le R$ and $v\in \cala_N$,
 $$
 \E \left (
  \sup_{0\le t \le T}
 \left (
 \|  u_{v}^\eps(t \wedge \tau^\eps_K ) 
 -  u_{v}( t \wedge \tau^\eps_K) \|^2
 \right )
 \right )
+2  \E \left ( 
\int_0^{T\wedge \tau^\eps_K}
\|  (-\Delta)^{\frac {\alpha}2}
( u^\eps_{v}(s)  -  u_{v} (s)
 ) \|^2  ds 
 \right )
 $$
 $$
 + 2^{1-p} \lambda_1
 \E\left (
 \int_0 ^{T\wedge \tau^\eps_K}
  \|  u^\eps_{v}(s)  -  u_{v} (s)
  \|^p_{L^p(\R^n)} ds
  \right )
  $$
  $$
\le  2 \eps c_5
 e^{2Tc_2 + 2N^2  }
 + 32 \sqrt{\eps} e^{2Tc_2 + 2N^2  }
  (K^2+c_3^2)c_5,
  $$
  where 
  $c_5=
  c_3
 +  2\| \sigma_1\|^2_{L^2(0,T; ,L^2(\R^n,l^2))}
 + c_4T$, which shows that
\be\label{cvs p20}
 \lim_{\eps \to 0}
 \sup_{\|u_0\|\le R}
 \sup_{v\in \cala_N}
 \E \left (
  \sup_{0\le t \le T}
 \left (
 \|  u_{v}^\eps(t \wedge \tau^\eps_K ) 
 -  u_{v}( t \wedge \tau^\eps_K) \|^2
 \right )
 \right )=0,
 \ee
 \be\label{cvs p21}
 \lim_{\eps \to 0}
 \sup_{\|u_0\|\le R}
 \sup_{v\in \cala_N}
   \E \left ( 
\int_0^{T\wedge \tau^\eps_K}
\|  (-\Delta)^{\frac {\alpha}2}
( u^\eps_{v}(s)  -  u_{v} (s)
 ) \|^2  ds 
 \right ),
 \ee
 and
 \be\label{cvs p22}
 \lim_{\eps \to 0}
 \sup_{\|u_0\|\le R}
 \sup_{v\in \cala_N} 
 \E\left (
 \int_0 ^{T\wedge \tau^\eps_K}
  \|  u^\eps_{v}(s)  -  u_{v} (s)
  \|^p_{L^p(\R^n)} ds
  \right )
  =0.
  \ee

By  \eqref{cvs p11}   we have, 
 for all  $\eps \in (0,1) $, 
 $\|u_0\|\le R$ and $v\in \cala_N$,
   \be\label{cvs p23}
 P(\tau^\eps_K <T)
\le 
 P \left (
 \sup_{t\in [0,T]}
 \| u^\eps_{v } (t) \|  \ge K  
 \right )
 \le
 {\frac 1{K^2}}
 \E \left ( 
 \| u^\eps_{v }   \|^2_{C([0,T], H)} 
 \right )
  \le
 {\frac {c_3}{K^2}}.
\ee 
By 
 \eqref{cvs p23} we see that
  for 
  every  $\eta>0$,
 $$
 \sup_{\|u_0\|\le R}
 \sup_{v\in \cala_N} 
 P 
\left (
\sup_{0\le t \le T}
 \|  u_{v}^\eps(t )
  -  u_{v}( t ) \| >\eta
\right )
$$
$$
 \le \sup_{\|u_0\|\le R}
 \sup_{v\in \cala_N} 
 P 
\left (
\sup_{0\le t \le T}
 \|  u_{v}^\eps(t )
  -  u_{v}( t ) \| >\eta,
  \  \tau^\eps_K =T
\right )
$$
$$
+
\sup_{\|u_0\|\le R}
 \sup_{v\in \cala_N} 
  P 
\left (
\sup_{0\le t \le T}
 \|  u_{v}^\eps(t )
  -  u_{v}( t ) \| >\eta,
  \  \tau^\eps_K <T
\right )
$$
$$
 \le \sup_{\|u_0\|\le R}
 \sup_{v\in \cala_N} 
 P 
\left (
\sup_{0\le t \le T}
 \|  u_{v}^\eps(t \wedge  \tau^\eps_K  )
  -  u_{v}( t \wedge  \tau^\eps_K) \| >\eta  
\right )
+
   {\frac {c_3}{K^2}}.
$$
$$
 \le \sup_{\|u_0\|\le R}
 \sup_{v\in \cala_N} 
 \eta^{-2}
\E
\left (
\sup_{0\le t \le T}
 \|  u_{v}^\eps(t \wedge  \tau^\eps_K  )
  -  u_{v}( t \wedge  \tau^\eps_K) \|^2 
\right )
+
   {\frac {c_3}{K^2}}.
$$

First  taking the limit as $\eps \to 0$  and then
as 
 $K\to \infty$,  by \eqref{cvs p20} we obtain
 for 
  every  $\eta>0$,
  \be \label{cvs p24}
\lim_{\eps \to 0}
 \sup_{\|u_0\|\le R}
 \sup_{v\in \cala_N} 
 P 
\left (
\sup_{0\le t \le T}
 \|  u_{v}^\eps(t )
  -  u_{v}( t ) \| >\eta
\right )
 =0.
 \ee

By   \eqref{cvs p23} we have,
for every $\eta>0$,
 $$
 \sup_{\|u_0\|\le R}
 \sup_{v\in \cala_N} 
 P 
\left (
 \int_0^T
 \|  u_{v}^\eps(t )
  -  u_{v}( t ) \|^p_{L^p(\R^n)} dt >\eta^p
\right )
$$
$$
\sup_{\|u_0\|\le R}
 \sup_{v\in \cala_N} 
 \le P 
\left (
 \int_0^T
 \|  u_{v}^\eps(t )
  -  u_{v}( t ) \| ^p_{L^p(\R^n)}  dt >\eta^p,
  \  \tau^\eps_K =T
\right )
$$
$$
+
\sup_{\|u_0\|\le R}
 \sup_{v\in \cala_N} 
  P 
\left (
 \int_0^T
 \|  u_{v}^\eps(t )
  -  u_{v}( t ) \| ^p_{L^p(\R^n)} dt >\eta^p,
  \  \tau^\eps_K <T
\right )
$$
$$
 \le 
 \sup_{\|u_0\|\le R}
 \sup_{v\in \cala_N} 
 P 
\left (
 \int_0^{T \wedge \tau_K^\eps}
 \|  u_{v}^\eps(t )
  -  u_{v}( t ) \| ^p_{L^p(\R^n)}  dt >\eta^p 
\right )
+
   {\frac {c_3}{K^2}}
$$
$$
 \le 
 \sup_{\|u_0\|\le R}
 \sup_{v\in \cala_N} 
 \eta^{-p}
\E
\left (
 \int_0^{T \wedge \tau_K^\eps}
 \|  u_{v}^\eps(t )
  -  u_{v}( t ) \| ^p_{L^p(\R^n)}  dt  
\right )
+
   {\frac {c_3}{K^2}},
$$
which together with
\eqref{cvs p22} implies that
\be\label{cvs p25}
 \lim_{\eps \to 0}\sup_{\|u_0\|\le R}
 \sup_{v\in \cala_N} 
 P 
\left (
 \int_0^T
 \|  u_{v}^\eps(t )
  -  u_{v}( t ) \|^p_{L^p(\R^n)} dt >\eta^p
\right )=0.
\ee
 Similarly, by \eqref{cvs p21} and \eqref{cvs p23} one can obtain
\be\label{cvs p26}
 \lim_{\eps \to 0}\sup_{\|u_0\|\le R}
 \sup_{v\in \cala_N} 
 P 
\left (
 \int_0^T
 \| (-\Delta)^{\frac \alpha{2}} (u_{v}^\eps(t )
  -  u_{v}( t )) \|^2  dt >\eta^2
\right ) =0.
\ee
Then \eqref{cvs p24}-\eqref{cvs p26} 
conclude the proof.
 \end{proof}

We are now ready to prove
Theorem \ref{main}.

{\bf Proof of Theorem \ref{main}}.
 By \eqref{comg0} and  Lemma \ref{cvs}
 we find that \eqref{convp}
 and \eqref{convp1} are satisfied, and thus
 Theorem \ref{main}
 follows from 
  Proposition \ref{fwuldp} 
  immediately.

Next, we discuss the
 Dembo-Zeitouni uniform
 LDP of solutions
 over compact initial data.
 To that end, we need the
 continuity of   level sets
 of the rate function
 as stated below.

 \begin{lem}
 	\label{levc}
 Let  \eqref{f1}-\eqref{f3}
 and \eqref{sig1}-\eqref{sig4}
 be fulfilled and $s\ge 0$.
 If   $u_{0,n} \to u_0$ in $H$, then 
   	$\{I_{ u_{0,n}}^s\}_{n=1}^\infty $ converges
 	to $I_{ u_0}^s$ in the Hausdorff
 	metric 
 	in $C([0,T], H)
 	\bigcap L^2(0,T; V)
 	\bigcap L^p(0,T; L^p(\R^n))$:
 \be\label{levc 1}
 	\sup_{\phi \in I^s_{u_0}}
 	\left (
 	{\rm dist}_{
 	C([0,T], H)
 	\bigcap L^2(0,T; V)
 	\bigcap L^p(0,T; L^p(\R^n))
 }\ 
 	(\phi ,\  I^s_{ u_{0,n}}) 
 	\right )
   \longrightarrow  0, \ \ \text{as } \ n\to \infty,
 \ee
 	and
 \be\label{levc 2}
 	\sup_{\phi \in I^s_{u_{0,n}}}
 	\left (
 	{\rm dist}_{
 		C([0,T], H)
 		\bigcap L^2(0,T; V)
 		\bigcap L^p(0,T; L^p(\R^n))
 	}\ 
 	\left (\phi , \ I^s_{ u_{0}}
 	\right )
 	\right )
   	\longrightarrow  0, \ \ \text{as } \ n\to \infty.
  \ee 
 \end{lem}
 
 \begin{proof}
 	Let   $u_v (\cdot, u_0)$  
 	be 
 	the solution of 
 	\eqref{contr1}--\eqref{contr2}.
 	If  $\phi \in I^s_{u_0}$,
 	then $\phi = u_v (\cdot, u_0)$
 for some  $v\in L^2(0,T; H)$
  with
 	${\frac 12}\int_0^T
 	\| v(s) \|^2 ds \le s$.  
 By Lemma \ref{cosol}
 	we  have
 	$$
 	{\rm dist}_{
 	C([0,T], H)
 	\bigcap L^2(0,T; V)
 	\bigcap L^p(0,T; L^p(\R^n))
 }\
 	(\phi , I^s_{ u_{0,n}})
 	$$
 	$$
 	\le 
 	{\rm dist}_{
 	C([0,T], H)
 	\bigcap L^2(0,T; V)
 	\bigcap L^p(0,T; L^p(\R^n))
 }
 \ 	  
 	( u_v(\cdot, u_{0} )  , \
 	u_v(\cdot, u_{0,n} ))
 	$$
 	$$
 	\le c_1
 	\left ( \| u_{0,n} -u_0\|
 	+\| u_{0,n} -u_0\|^{\frac p2}
 	\right ),
 	$$
 	where $c_1=c_1(T)>0$.
 	Since $\phi \in I^s_{u_0}$ 
 	is arbitrary, we get  
 	$$
 	\sup_{\phi \in I^s_{u_0}}
 	{\rm dist}_{
 	 C([0,T], H)
 	 \bigcap L^2(0,T; V)
 	 \bigcap L^p(0,T; L^p(\R^n))
 }\
 	(\phi,  I^s_{ u_{0,n}})
 \le c_1
 \left ( \| u_{0,n} -u_0\|
 +\| u_{0,n} -u_0\|^{\frac p2}
 \right ),
  $$
 	from which
 	 \eqref{levc 1}
 	follows due to the assumption
 	that  $u_{0,n}
 	\to u_0$ in $H$.
 	
 	The convergence \eqref{levc 2}
 	can be proved by a similar
 	argument, and the details 
 	are omitted.
  \end{proof}

 We now present the 
   Dembo-Zeitouni uniform LDP
 of solutions of \eqref{intr1}-\eqref{intr3}
 over compact initial data.

 \begin{thm}\label{main1}
 	If    \eqref{f1}-\eqref{f3}
 	and \eqref{sig1}-\eqref{sig4}
 	are fulfilled, then the
 	solutions 
 	$u^\eps(\cdot, u_0) $,
 	${0<\eps\le 1}$,
 	 of   
 	   \eqref{intr1}-\eqref{intr3}
 	satisfies  
 	the   Dembo-Zeitouni
 	uniform LDP
 	on $ 	C([0,T], H)
 	\bigcap L^2(0,T; V)
 	\bigcap L^p(0,T; L^p(\R^n))$
 	with rate function $I_{u_0}$, 
 	uniformly with respect to 
 	$u_0$
 over   compact subsets of $H$:
 	
 	\begin{enumerate} 
 		\item  For 
 		every open subset $G$ of $ 	C([0,T], H)
 		\bigcap L^2(0,T; V)
 		\bigcap L^p(0,T; L^p(\R^n))$ and
 		 every 
 		compact subset  $\calz $ of $H$,
 		$$
 		\liminf_{\eps \to 0}
 		\inf _{u_0 \in \calz}
 		\left (
 		\eps \ln 
 		P(u^\eps (\cdot,  u_0)
 		\in G )
 		\right )
 		\ge -\sup_{u_0 \in \calz}
 		\inf_{\phi \in G} I_{ u_0} (\phi ).
 		$$
 		
 		\item  For  
 		every closed subset $F$ of $ 	C([0,T], H)
 		\bigcap L^2(0,T; V)
 		\bigcap L^p(0,T; L^p(\R^n))$
 		and every
 		compact subset  $\calz $ of $H$,
 		$$
 		\limsup_{\eps \to 0}
 		\sup_{u_0 \in \calz}
 		\left (
 		\eps \ln 
 		P(u^\eps (\cdot,  u_0)
 		\in F )
 		\right )
 		\le  -\inf_{u_0 \in \calz}
 		\inf_{\phi \in F} I_{ u_0} (\phi ).
 		$$
 	\end{enumerate}
 \end{thm}
 
 \begin{proof}
 This is an immediate consequence
 of    Proposition
 \ref{eqldp}, 
 Theorem \ref{main} and 
 Lemma \ref{levc}.
 \end{proof}

 \section*{Competing interests and declarations}
 I declare that the author has
 no competing interests  
 or other interests that might be
 perceived to influence the results and/or 
 discussion reported in this paper.
 
  \section*{ 
 Data availability statements
}
  This manuscript
 has no associated data.
% \section*{Author Contributions Statement}
% 
% 
% \section*{Acknowledgements} 
% 

\end{document}